\documentclass[12pt]{amsart}

\usepackage{hyperref}
\hypersetup{
    colorlinks=true,
    linkcolor=blue,
    filecolor=magenta,
    urlcolor=cyan,
    citecolor=blue,
}
\usepackage{thmtools,thm-restate}

\usepackage{amsmath,amsxtra,amssymb,latexsym,epsfig,amscd,amsthm,subfigure,fancybox}
\usepackage[mathscr]{eucal}
\usepackage{graphicx}
\usepackage{epsfig}
\usepackage{epstopdf}
\usepackage{cases}
\usepackage{multicol,xcolor}
\usepackage{enumerate}
\newcommand{\lf}{\lfloor}
\newcommand{\rf}{\rfloor}

\newcommand{\xx}{x^{\eps,\delta}}
\newcommand{\yy}{y^{\eps,\delta}}
\newcommand{\zz}{z^{\eps,\delta}}
\newcommand{\ale}{\alpha^{\eps}}
\newtheorem{thm}{Theorem}[section]
\newtheorem {asp}{Assumption}[section]
\newtheorem{lm}{Lemma}[section]
\newtheorem{rmk}{Remark}[section]

\newtheorem{prop}[thm]{Proposition}
\theoremstyle{definition}
\newtheorem*{conv*}{Convention}

\theoremstyle{remark}

\numberwithin{equation}{section}

\DeclareMathOperator{\inte}{Int}

\newcommand{\eps}{\varepsilon}

\newcommand{\M}{\mathcal{M}}
\newcommand{\K}{\mathcal{K}}
\newcommand{\F}{\mathcal{F}}

\newcommand{\E}{\mathbb{E}}

\newcommand{\N}{\mathbb{N}}
\newcommand{\PP}{\mathbb{P}}

\newcommand{\R}{\mathbb{R}}

\newcommand{\cd}{(\cdot)}

\numberwithin{equation}{section}
\newcommand{\1}{\boldsymbol{1}}

\newcommand{\wdt}{\widetilde}

\newcommand{\bed}{\begin{displaymath}}
\newcommand{\eed}{\end{displaymath}}
\newcommand{\bea}{\bed\begin{array}{rl}}
\newcommand{\eea}{\end{array}\eed}
\newcommand{\ad}{&\!\!\!\disp}
\newcommand{\aad}{&\disp}
\newcommand{\barray}{\begin{array}{ll}}
\newcommand{\earray}{\end{array}}

\newcommand{\dist}{{\rm dist}}
\def\disp{\displaystyle}

\def\trace{\text{trace}}
\def\bar{\overline}
\def\a.s{\text{\;a.s.\;}}

\begin{document}
\title[Dynamical
systems with fast switching and slow diffusion]{Dynamical Systems
under Random Perturbations with Fast Switching and Slow Diffusion: Hyperbolic Equilibria and Stable Limit Cycles}\thanks{The work of N.H. Du was supported in part by NAFOSTED n$_0$ 101.02 - 2011.21; the work of A. Hening was supported in part by the
National Science Foundation under grant DMS-1853463; the work of D. Nguyen was
supported in part by the National Science Foundation under grant DMS-1853467; the research of G. Yin was supported in part by the National Science Foundation under grant
DMS-1710827.}
\thanks{We thank an anonymous referee who read the original version of the manuscript and suggested to study the more general setting of hyperbolic equilibria. This has significantly improved the paper.}

\author[N. H. Du]{Nguyen H. Du}
\address{Department of Mathematics, Mechanics and
Informatics\\
Hanoi National University\\
334 Nguyen Trai\\
Thanh Xuan\\
Hanoi\\
Vietnam}
\email{dunh@vnu.edu.vn}

\author[A. Hening]{Alexandru Hening }
\address{Department of Mathematics\\
Texas A\&M University\\
Mailstop 3368\\
College Station, TX 77843-3368\\
United States
}
\address{Department of Mathematics\\
Tufts University\\
Bromfield-Pearson Hall\\
503 Boston Avenue\\
Medford, MA 02155\\
United States
}
\email{al.hening@gmail.com}

\author[D. Nguyen]{Dang H. Nguyen }
\address{Department of Mathematics \\
University of Alabama\\
345 Gordon Palmer Hall\\
Box 870350 \\
Tuscaloosa, AL 35487-0350 \\
United States}
\email{dangnh.maths@gmail.com}

\author[G. Yin]{George Yin}
\address{Department of Mathematics\\
Wayne State University\\
Detroit, MI 48202\\
United States
}
\email{gyin@wayne.edu}

\maketitle

\begin{abstract}
This work is devoted to the study of long-term qualitative behavior of randomly perturbed dynamical systems.
The focus is on
certain
stochastic differential equations (SDE) with Markovian switching,
 when the switching is fast varying and the diffusion (white noise) is slowly changing.
Consider the system
\[
dX^{\eps,\delta}(t)=f(X^{\eps,\delta}(t), \alpha^\eps(t))dt+\sqrt{\delta}\sigma(X^{\eps,\delta}(t), \alpha^\eps(t))dW(t)
, \ X^{\eps,\delta}(0)=x, \ \alpha^\eps(0)=i,
\]
where $\alpha^\eps(t)$ is a
finite state Markov chain with irreducible generator $Q=(q_{\iota\ell})$.
The relative changing rates of
the switching and the diffusion are
highlighted by two small parameters $\eps$ and $\delta$.
Associated with the above stochastic differential equation, there is an averaged ordinary differential equation (ODE)
\[
d\bar X(t)=\bar f(\bar X(t))dt, \ \bar X(0)=x,
\]
where $\bar f(\cdot)=\sum_{\iota=1}^{m_0}f(\cdot, \iota)\nu_\iota$ and $(\nu_1,\dots,\nu_{m_0})$ is the unique stationary distribution of the Markov chain with generator $Q$.
Suppose that for each pair $(\eps,\delta)$, the process has an invariant probability measure $\mu^{\eps,\delta}$,
and that the averaged ODE has a limit cycle in which there is an averaged occupation measure $\mu^0$ for the averaged equation.
 It is proved in this paper that under weak conditions,
 if $\bar f$ has finitely many stable or hyperbolic fixed points, then $\mu^{\eps,\delta}$ converges weakly to $\mu^0$ as $\eps \to 0$ and $\delta \to 0$.
Our results generalize to the setting where the switching process $\alpha^\eps$ is state-dependent in that
\[
\PP\{\alpha^\eps(t+\Delta)=\ell | \alpha^\eps=\iota, X^{\eps,\delta}(s),\alpha^\eps(s), s\leq t\}=q_{\iota \ell}(X^{\eps,\delta}(t))\Delta+o(\Delta),~~ \iota\neq \ell
\]
as long as the generator $Q(\cdot)=(q_{\iota \ell}(\cdot))$ is locally bounded, locally Lipschitz, and irreducible for all $x\in\R^d$. Finally, we provide two examples in two and three dimensions in order to showcase our results.

\end{abstract}

\tableofcontents

\section{Introduction}\label{sec:int}
Natural phenomena are almost always influenced by different types of random noise. In order to better understand the world around us, it
is important
%therefore makes sense
to study random perturbations of dynamical systems.
In the continuous dynamical systems setup,
the focus then shifts from the study of the behavior of deterministic differential equations to that of
differential equations with switching
%(or difference equations) to the study of different discrete (Markov chains, stochastic difference equations) or continuous-time Markov processes (stochastic differential equations,
(piecewise deterministic Markov processes) or stochastic differential equations with switching.
The long-term behavior of these systems can be analyzed by a careful study of the ergodic properties of the induced Markov processes.

Quite often, the ``white noise'' in the system is small compared to the deterministic component. In
%these
such cases, one is usually interested in knowing how well the deterministic system approximates the stochastic one. It is common to model continuous-time phenomena by stochastic differential equations of the type

 \begin{equation}\label{eq1.1}			
dx^\delta(t)=f(x^\delta(t))dt+\sqrt{\delta}\sigma(x^\delta(t))dW(t),
\end{equation}
where $f\cd$ and $\sigma\cd$ are sufficiently smooth
%enough
functions, $W(\cdot)$ is a standard $m$-dimensional Brownian motion, and $\delta>0$ is a small parameter.
If we let $\delta\to  0$, one would expect that the solutions of \eqref{eq1.1} converge, in an appropriate sense, to
%the solutions
that
of a deterministic differential equation.

Different aspects
%Versions
of the problem have been studied extensively starting with Freidlin and Wentzell \cite{FW70, FW98}, Fleming \cite{WF}, Kifer \cite{K81} and Day \cite{D82}.
If the process $x^\delta(t)$ has a unique ergodic probability measure $\mu^\delta$ for each $\delta>0$ and  the origin of the corresponding deterministic ODE
\begin{equation}\label{eq1.2}
dx=f(x)dt,
\end{equation}
is a globally asymptotic equilibrium point,
Holland established in \cite{CH1} asymptotic expansions of the expectation of
the underlying functionals with respect to the unique ergodic probability measures $\mu^\delta$.
In addition, in \cite{CH}, Holland considered the case when the ODE \eqref{eq1.2} has an asymptotically stable limit cycle and proved the weak convergence of the family $(\mu^\delta)_{\delta>0}$ to the unique stationary distribution that is concentrated on the limit cycle of the process from \eqref{eq1.2}.

Our interest in the current problem stems from applications
in ecology. Quite often, one models the dynamics of populations with continuous-time processes. This way we inherently assume that organisms can respond instantaneously to changes in the environment. However, in some cases the dynamics are better described by discrete-time models, in which demographic decisions are not made continuously. In order to model more complex systems, one has to analyze `hybrid' systems where both continuous and discrete dynamics coexist. Such systems arise naturally in ecology, engineering, operations research, and physics as well as in emerging applications in wireless communications, internet traffic modeling, and financial engineering; see \cite{YZ} for more references.

Recently, there has been renewed interest in studying piecewise deterministic
Markov
processes (PDMP) \cite{D84}. One may describe a PDMP by the use of a two component process. The first component is a continuous state process represented by the solution of a deterministic differential equation, whereas the second component is a discrete event process taking values in a finite set. This discrete event process is often modeled as a continuous-time Markov chain with a finite state space. At any given instance, the Markov chain takes a value (say $i$ in the state space), and the process sojourns in state $i$ for a random duration. During this period, the continuous state follows the flow given by a differential equation associated with $i$. Then at a random instance, the discrete event switches to another state $j\not =i$. The Markov chain sojourns in $j$ for a random duration, during which, the continuous state follows another flow associated with the discrete state $j$.

A careful study of such
processes has recently led to a
better understanding of
predator-prey communities where the predator evolves much faster than the prey \cite{C16} and for a possible explanation of how the competitive exclusion principle from ecology, which states that multiple species competing for the same number of small resources cannot coexist, can be violated because of switching \cite{BL16,HN18}.

It is natural to study the SDE counter-part of PDMP, that is, SDE with switching. Similarly to the piecewise deterministic Markov processes mentioned in the previous paragraph,
in this setting one follows a specific system of SDE for a random time after which the discrete event switches to another state, and the process is governed by a different system of SDE.
The resulting stochastic process has a discrete component (that switches among a finite number of discrete states) and a continuous component (the solution of SDE associate with
each fixed discrete event state).
We refer the reader to \cite{YZ} for an introduction to SDEs with switching.
Most of the work inspired by Freidlin and Wentzell has been concerned with local phenomena that involve the exit times and exit probabilities from neighbourhoods of equilibria.
One usually uses the theory of large deviations to analyze the exit problem from the domain of attraction of a stable equilibrium point.
Nevertheless, in applications such as those arising in ecology, one faces even more complex
% There are more complicated
situations such as the one treated in this paper, in which large deviation techniques are not applicable,
%sufficient,
and one needs to analyze the distributional scaling limits for the exit distributions \cite{B11}.
There have been some previous important studies for multiscale systems with fast and slow scales \cite{DS12, DSW12}. These previous papers have looked at large deviations in the related setting where one has a slow diffusion and the coefficients are fast oscillating. However, in contrast to our framework, the fast oscillations come from having periodic coefficients and inclusion of
%introducing
a factor $\frac{1}{\eps}$ into the periodic component of the coefficients. The way the fast oscillations are introduced in these previous papers is similar to how it is done when one does stochastic homogenization. In the present paper, the switching comes from a discrete random process $\alpha^\eps$.

In this paper, we consider dynamical systems represented by switching diffusions, in which
the switching is rapidly varying whereas the diffusion is slowly changing.
To be more precise, let $(\Omega, \F, \{\F_t\}, \PP)$ be a filtered probability space satisfying the usual conditions.
Consider the process $(X^{\eps,\delta})_{t\geq 0}$ defined by
\begin{equation}\label{eq2.1}
dX^{\eps,\delta}(t)=f(X^{\eps,\delta}(t), \alpha^\eps(t))dt+\sqrt{\delta}\sigma(X^{\eps,\delta}(t), \alpha^\eps(t))dW(t)
, \ X^{\eps,\delta}(0)=x, \ \alpha^\eps(0)=i,\end{equation}
 where $W(t)$ is an $m$-dimensional
 standard Brownian motion,  $\alpha^\eps(t)$ is
% an independent Markov chain that has
a finite-state Markov chain that is independent of the Brownian motion and that has a
 state space $\M=\{1,..., m_0\}$ and generator $Q/\eps=\big(q_{ij}/\eps\big)_{m_0\times m_0}$,
 $X^{\eps,\delta}$ is an $\R^d$-valued process, $f: \R^d\times\M\to\R^d, \sigma: \R^d\times\M\to\R^{d\times m},$ and
 $\eps, \delta>0$ are two small parameters.
 We assume that the matrix $Q$ is irreducible. The irreducibility of $Q$ implies that the Markov chain associated
 with $Q$, which will be denoted by $(\tilde \alpha(t))_{t\geq 0}$,
  is ergodic thus has a unique stationary distribution $(\nu_1,\dots,\nu_{m_0})$.
  We denote by $X^{\eps,\delta}_{x, i}(t)$
  the solution of \eqref{eq2.1} at time $t\geq0$ when the initial value is $(x,i)$ and by $\alpha^\eps_i(t)$ the Markov chain started at $i$.

Let us explore, intuitively, what happens when $\eps$ and $\delta$
are very small.
In this setting, $\alpha^\eps(t)$ converges very fast to its stationary distribution $(\nu_1,\dots,\nu_{m_0})$ while the diffusion is asymptotically small.
%very small.
As a result, on each finite time interval $[0,T]$ for $T>0$, a solution of equation \eqref{eq2.1} can be approximated by the solution $\bar X_x(t)$ to
\begin{equation}\label{eq2.2}
d\bar X(t)=\bar f(\bar X(t))dt, \  \bar X(0)=x,
\end{equation}
where $\bar f(x)=\sum_{i=1}^{m_0}f(x, i)\nu_i$.

However, if in lieu of a finite time horizon,
%instead
we look at the process on
the infinite time horizon
%interval
$[0,\infty)$,
it is not clear that $\bar X_x(t)$ is a good approximation. Suppose that equation \eqref{eq2.2} has a stable limit cycle. A natural question is whether the invariant probability measures $(\mu^{\eps,\delta})$ of the processes \eqref{eq2.1} converge weakly as $\eps\to  0$ and $\delta\to  0$, to the measure concentrated on the limit cycle. This is the main problem that we address in the current paper. In order to do this, we
substantially extend the results of \cite{CH} by considering the presence of both small diffusion and rapid switching.
Because of the presence of both the switching and the diffusion we need to develop new mathematical techniques.
In addition, even if there is no switching and we are in the SDE setting of \cite{CH},
% It will be apparent that even in this situation
our assumptions are weaker than those used in \cite{CH}.
%\begin{rmk}
%One might be interested in the following natural generalization of the setting presented above. The switching process $\alpha^\eps$ can be state dependent, that is
%\[
%\PP\{\alpha^\eps(t+\Delta)=j~|~\alpha^\eps=i, X^{\eps,\delta}(s),\alpha^\eps(s), s\leq t\}=q_{ij}(X^{\eps,\delta}(t))\Delta+o(\Delta).
%\]
%As long as the generator $Q(x)=(q_{ij}(x))$ is irreducible for each $x\in\R^d$ one can show that on each finite time interval $[0,T]$ one can approximate the process from \eqref{eq2.1} if $\eps,\delta$ are small by
%\[
%d\bar X = \bar f(\bar X(t))dt, \bar X(0)=x,
%\]
%where $\bar f(x)=\sum_{i=1}^{m_0} f(x,i)\nu_i(x)$ and $(\nu_1(x),\dots,\nu_{m_0}(x))$ is the stationary distribution of a Markov chain with generator $Q(x)=(q_{ij}(x))$. We will explain through a sequence of remarks that our results hold for this generalization.
%\end{rmk}

The rest of the paper is organized as follows.
The main assumptions and results are given in Section \ref{sec:2}.
To provide an insight of the  proofs and to connect different parts of the arguments
 so as to provide something like a ``road map'', Section \ref{subsec:main} presents the
main ideas of proofs.
In Section \ref{sec:3}, we estimate the exit time of the solutions of \eqref{eq2.1} from neighborhoods around the stable manifolds of the equilibria of $\bar f$. The proofs of the main results are presented in Section \ref{sec:4}. In Section
\ref{sec:5}, we apply our results to a general predator-prey model. In addition to showcasing
our result in a specific setting, the proofs from Section \ref{sec:5} are interesting on their own right as they are quite technical and require the development of new tools. Finally, in Section \ref{sec:6}, we provide some numerical examples to illustrate our results from the predator-prey setting in Section \ref{sec:5}.

\subsection{Assumptions and main results}\label{sec:2}
We denote by $A'$ the transpose of a matrix $A$, by $|\cdot|$ the Euclidean norm of vectors in $\R^d$, and by $\|A\|:=\sup\{|Ax|: x\in\R^d, |x|=1\}$ the operator norm of a matrix $A\in \R^{d\times d}$. We also define $a\wedge b:=\min\{a, b\}$, $a\vee b:=\max\{a, b\}$, and the closed ball of radius $R>0$ centered at the origin $B_R:=\{x\in\R^d:|x|\leq R\}$.

We recall some definitions due to Conley \cite{C78}. Suppose we are given a flow $(\Phi_t(\cdot))_{t\in\R}$. A compact invariant set $K$ is called \textit{isolated} if there exists a neighborhood $V$ of $K$ such that $K$ is the maximal compact invariant set in $V$. A collection of nonempty sets $\{M_1,\dots,M_k\}$ is a \textit{Morse decomposition} for a compact invariant set $K$ if $M_1,\dots,M_k$ are pairwise disjoint, compact, isolated sets for the flow $\Phi$ restricted to $K$ and the following properties hold: 1) For each $x\in K$ there are integers $l=l(x)\leq m=m(x)$ for which the \textit{alpha limit set} of $x$, $\hat\alpha(x)=\bigcap_{t\leq 0} \overline{\{\Phi_s(x), s\in(-\infty,t] \}}$, satisfies $\hat\alpha(x)\subset M_l$ and the \textit{omega limit set} of $x$, $\hat\omega(x):=\bigcap_{t\geq 0} \overline{\{\Phi_s(x), s\in[t,\infty) \}}$, satisfies $\hat\omega(x)\subset M_m$ 2) If $l(x)=m(x)$ then $x\in M_l=M_m$.

\begin{asp}\label{asp1}
We impose the following assumptions for the processes modeled
%modelled
by the systems \eqref{eq2.1} and \eqref{eq2.2}.
\begin{enumerate}[(i)]
  \item
  For each $i\in \M$,
  $f(\cdot, i)$ and  $\sigma(\cdot, i)$ are locally Lipschitz continuous.
  \item There is an $a>0$ and a twice continuously differentiable
   real-valued, nonnegative function $\Phi(\cdot)$ satisfying $\lim\limits_{R\to\infty}\inf\{\Phi(x): |x|\geq R\}=\infty$ and $(\nabla \Phi)'(x)f(x,i) \leq a(\Phi(x)+1)$,
   for all $(x, i)\in\R^d\times\M$.
\item The vector field $\bar f(\cdot)$ is $C^1$ and it has finitely many equilibrium points $\{x_1,\dots,x_{n_0-1}\}$  and a unique limit cycle $\Gamma$. The equilibrium points are  hyperbolic points.
\item
There exists a Morse decomposition $\{M_1, M_2,\cdots, M_{n_0}\}$ of the flow associated with $\bar f$ such that
$M_{n_0}=\Gamma$ is the limit cycle and for any $i<n_0$ we have $M_i=\{x_i\}$ where $x_i$ is an equilibrium point.
\item There exists $\eps_0>0$ such that for all $0<\eps<\eps_0$, the system \eqref{eq2.1} has a unique solution.
Furthermore, for any $0<\eps<\eps_0$, the process $(X^{\eps,\delta}(t), \alpha^{\eps}(t))$ has the strong Markov property and has an invariant probability measure $\mu^{\eps,\delta}$.
The family $(\mu^{\eps,\delta})_{0<\eps<\eps_0}$ is tight, i.e., for any $\gamma>0$ there exists an $R=R_\gamma>0$ such that $\mu^{\eps,\delta}(B_R\times\M)>1-\gamma$ for all $0<\eps<\eps_0$.
\end{enumerate}
\end{asp}

\begin{rmk}
We note that using Assumption (ii), we can work in a compact state space $\K\subset \R^n$ if the diffusion term from \eqref{eq2.1} is zero.  Assumptions (i) and (ii) are needed in order to deduce the existence and boundedness of a unique solution to equation \eqref{eq2.1} in the absence of the diffusion term. Assumption (iv) is used to make sure that there exist no heteroclinic cycles.

Assumption (v) ensures that \eqref{eq2.1} has a unique solution that is strong Markov. Sufficient conditions that imply uniqueness and the strong Markov property exist in the literature \cite{MY, YZ}.
\end{rmk}

\begin{rmk}
In \cite{CH} the author studied \eqref{eq1.1} under the assumptions that
\begin{enumerate}
  \item [(A1)] $f,\sigma\in C^2(\R^d)$.
  \item [(A2)] The system \eqref{eq1.2} has a unique limit cycle.
  \item [(A3)] There exists at most a finite number of equilibria $x^*$ of $f$. At each equilibrium the Jacobian matrix has only positive real parts and the matrix $\sigma'\sigma$ is positive definite.
  \item [(A4)] For any compact set $B$ not containing equilibria and any $u>0$ there exists $T>0$ such that if $x\in B$, then
  \[
  d(x^0_x(t),\Gamma)<u,~\text{for}~t\geq T.
  \]
  \item [(A5)] There exists $\delta_0>0$ such that for $0<\delta<\delta_0$ the stochastic differential equation has a unique ergodic measure $\mu_\delta$. Furthermore, the family $(\mu_{\delta})_{0<\delta<\delta_0}$ is tight in $\R^d$.
\end{enumerate}
Our work generalizes \cite{CH} significantly in the following aspects.
First, we work with two types of randomness - one comes from the diffusion term and the other from the switching mechanism. Second, Assumption \ref{asp1} (i) is weaker than (A1). Third, we can have any hyperbolic fixed points whereas assumptions (A3)-(A4) imply that all fixed points are sources and the deterministic system converges uniformly to the limit cycle. In addition, we do not need $\sigma'\sigma$ to be positive definite at the equilibria.
\end{rmk}

\begin{rmk}
There are several papers, which look at the exit time asymptotics near hyperbolic fixed points of small perturbations of dynamical systems \cite{K81, B08,B11}. In contrast to
%We do not assume like in
these papers in which the noise is uniformly elliptic,   we have to deal with the additional complications of a stable limit cycle as well as the switching due to $\alpha^\eps$.
\end{rmk}

Let $T_\Gamma>0$ be the period of the limit cycle $\Gamma$. We can define a probability measure $\mu^0$, which is independent of the starting point $y\in \Gamma$, by
\begin{equation}\label{e:mu0}
\mu^0(\cdot)=\dfrac1{T_\Gamma}\int_0^{T_\Gamma} \1_{( \cdot)}(\bar X_y(s))ds,
\end{equation}
where $\bar X_y(t)$ is the solution to equation \eqref{eq2.2} starting at $ \bar X(0)=y$ and $\1_{\{\cdot\}}$ is the indicator function.
The measure $\mu^0\cd$ is the averaged occupation measure of the process $\bar X$ restricted to the limit cycle $\Gamma$. Throughout the paper, we assume that $\delta$ depends on $\eps$, i.e. $\delta= \delta(\eps)$, and $\lim\limits_{\eps\to  0}\delta (\eps)=0$. We will investigate the asymptotic behavior of the invariant probability measures $\mu^{\eps,\delta}$ as $\eps \to  0$ in the following three cases:
\begin{equation}\label{eq:ep-dl}
\lim\limits_{\eps\to  0}\dfrac{\delta}\eps=\left\{\begin{array}{ll}l\in(0,\infty), &\mbox{ case 1}\\0, &\mbox{ case 2}\\\infty, &\mbox{ case 3}.\end{array}\right.
\end{equation}
The multi-scale modeling approach we use is similar to the one from \cite{HY14}.
\begin{asp}\label{asp2}We
impose  additional conditions corresponding to the cases from \eqref{eq:ep-dl}.
\begin{enumerate}
\item[1)] Suppose $\lim_{\epsilon\to  0}\frac{\delta}{\eps}=l\in (0,\infty)$. For any equilibrium $x^*$ of $\bar f$ there exists $i^*\in \M$ such that $\beta'f(x^*, i^*)\ne 0$ or $\beta'\sigma(x^*, i^*)\sigma'(x^*,i^*)\beta\ne0$ where $\beta$ is a normal vector of the stable manifold of \eqref{eq2.2} at $x^*$.
\item[2)] Suppose $\lim_{\epsilon\to  0}\frac{\delta}{\eps}=0$. For any equilibrium $x^*$ of $\bar f$ there exists $i^*\in \M$ such that $\beta'f(x^*, i^*)\ne 0$ where $\beta$ is a normal vector of the stable manifold of \eqref{eq2.2} at $x^*$.
\item[3)] Suppose $\lim_{\epsilon\to  0}\frac{\delta}{\eps}=\infty$. For any equilibrium $x^*$ of $\bar f$, there exists $i^*\in \M$ such that $\beta'\sigma(x^*, i^*)\sigma'(x^*,i^*)\beta\ne 0$ where $\beta$ is a normal vector of the stable manifold of \eqref{eq2.2} at $x^*$.
\end{enumerate}
Note that when $x^*$ is a source, the stable manifold is $0$-dimensional and any non-zero vector can be considered as a normal vector.
\end{asp}

  The intuition for the conditions of Assumption \ref{asp2} is the following. In case 2, since $\delta$ tends to $0$ much faster than $\eps$,
  for sufficiently small $\delta$,
the behavior of $X^{\eps,\delta}(t)$ will be close to the process $\xi^{\eps}(t)$ defined by
\begin{equation}\label{eq2.3}
d\xi^\eps(t)=f\big(\xi^\eps(t),\alpha^\eps(t)\big)dt.
\end{equation}
We denote from now on by  $\xi^{\eps}_{x, i}(t)$ the solution of \eqref{eq2.3} at time $t\geq 0$ if the initial condition is $(x,i)$.

If for each $i\in \M$, $f(x^*, i)=0$  at a equilibrium $x^*$ of $\bar f$, the Dirac mass function at $x^*$, $\boldsymbol\delta_{x^*}$, will be an invariant
measure for $\xi^\eps(t)$. Because of this, the sequence of invariant probability measures $(\mu^{\eps,\delta})$ (or one of its subsequences) may converge to $\boldsymbol\delta_{x^*}$. In order to have the weak convergence of $(\mu^{\delta, \eps})_{\eps>0}$ to the measure $\mu^0$, we need to assume that there is an $i^*\in\M$ such that $\beta'f(x^*, i^*)\ne0$ where $\beta$ is a normal vector of the stable manifold of \eqref{eq2.2} at $x^*$. This guarantees that the process from \eqref{eq2.3} gets pushed away from the equilibrium $x^*$ and away from the stable manifold (where it could get pushed back towards the equilibrium).

In case 3, the switching is very fast compared to the diffusion term, so for small $\eps$ the process will behave like
\[
d \eta^\eps(t) = \bar f(\eta^\eps(t))dt + \sqrt{\delta}\bar\sigma(\eta^\eps(t))d\bar W(t).
\]
where
$\bar\sigma(x)=\left(\sum_{i\in\M}\sigma(x,i)\sigma^\top(x,i)\nu_i\right)^{\frac12}$  and $\bar W(t)$ is an independent $n$-dimensional Brownian motion.

 In order for the limit of $(\mu^{\eps,\delta})$ not to put mass on the equilibrium $x^*$ of $\bar f$, we need to suppose that there exists an $i^*\in\M$ such that $\beta'\sigma(x^*, i^*)\ne0$ where  $\beta$ is a normal vector of the stable manifold of \eqref{eq2.2} at $x^*$

 For case 1,  since both the switching and the diffusion are on a similar scale, we need to assume that for each equilibrium $x^*$ of $\bar f$ there is $i^*\in\M$ satisfying either $\beta'\sigma(x^*, i^*)\ne0$ or $\beta'f(x^*, i^*)\ne0$.

The next theorem is the main result of this paper.
\begin{restatable}{thm}{main}\label{t:main}
Suppose Assumptions \ref{asp1} and \ref{asp2} hold. The family of invariant probability measures $(\mu^{\eps,\delta})_{\eps>0}$ converges weakly to the measure $\mu^0$ given by \eqref{e:mu0} in the sense that for every bounded and continuous function $g:\R^d\times\M\to \R$,
$$\lim\limits_{\eps\to0}\sum_{i=1}^{m_0}\int_{\R^d} g(x, i)\mu^{\eps,\delta}(dx, i)=\dfrac1{T_\Gamma}\int_{0}^{T_\Gamma}\bar g(\bar X_y(t))dt,$$
where $T_\Gamma$ is the period of the limit cycle, $y\in\Gamma$ and $\bar g(x)=\sum_{i\in\M}g(x, i)\nu_i$.

\end{restatable}
\begin{rmk}\label{r:state}
	Given that the switching component $\alpha^\eps$ is state-dependent with generator $Q(x)=(q_{ij}(x))_{\M\times\M}, x\in \R^d$,
 Theorem \ref{t:main} still holds with
 $\bar f(x)=\sum_{i=1}^{m_0} f(x,i)\nu_i(x)$ and $(\nu_1(x),\dots,\nu_{m_0}(x))$ is the stationary distribution of a Markov chain with generator $Q(x)=(q_{ij}(x))$
  as long as $Q(x)$ is bounded and satisfies the  following conditions:
\begin{itemize}
	\item For all $i$ the functions $q_{ii}(\cdot)$ and $\frac{q_{ij}(\cdot)}{q_{ii}(x\cdot)}$ are Lipschitz continuous.
     \item If $q_{ij}(x)>0$ for some $x\in\R^d$ then $\inf_{x\in\R^d} \frac{q_{ij}(x)}{|q_{ii}(x)|}>0$.
     \item For all $i$ we have $\inf_{x\in\R^d}|q_{ii}(x)|>0$.
     \item $\inf_{x\in\R^d} \hat q^{(m_0)}_{ij}(x)>0$ where $\hat Q(x)=(0\vee q_{ij}(x))_{\M\times\M}$, and $(\hat q^{(m_0)}_{ij}(x))$ is the $m_0$-power of $\hat Q(x)$%. \textcolor{red}{For all $k$? Or there is some $k$? $Q^k$ is the $k$th power of the matrix $Q$ right?}
\end{itemize}
  We explain how one can do this in Remark \ref{r:dens}.
\end{rmk}

\subsection{An application of Theorem \ref{t:main}} We will exhibit an example where the result of Theorem \ref{t:main} applies. Recently there has been renewed
interest in stochastic population dynamics \cite{HN16, BL16, B18, HN18}.
Suppose we have a predator-prey system of the form
\begin{equation}\label{ex00}
\left\{\begin{array}{lll}
\disp {d \over dt}{x}(t)=x(t)\left[ a- bx(t)-y(t)h(x(t), y(t))\right]\\
\disp {d \over dt}
y(t)=y(t)\left[- c- d y(t)+x(t)fh(x(t), y(t))\right]
.\end{array}\right.
\end{equation}
Here $x(t), y(t)$ denote the densities of prey and predator
at time $t\geq 0$, respectively; $a, b, c, d, f>0$ describe the per-capita birth/death and competition rates, and $xh(x,y), yh(x,y)$ are the functional responses of the predator and the prey. For instance, if $h(x, y)$ is constant, the model is the classical Lotka-Volterra one \cite{L25, V28, GH79}.
If $$h(x, y)=\dfrac{m_1}{m_2+m_3x+m_4y},$$ the functional response is of Beddington-DeAngelis type \cite{CC01}.
The setting of \eqref{ex00} is very general and encompasses many of the models used in the ecological literature.

We explore what happens in the fast-switching slow-noise limit for the following noisy extension of \eqref{ex00}
 \begin{equation}\label{ex1}
\left\{\begin{array}{lll}d{X^{\eps,\delta}}(t)=X^{\eps,\delta}(t)\varphi\big(X^{\eps,\delta}(t), Y^{\eps,\delta}(t), \alpha^\eps(t))dt+\sqrt{\delta}\lambda(\alpha^\eps(t))X^{\eps,\delta}(t)dW_1(t)\\
d{Y^{\eps,\delta}}(t)=Y^{\eps,\delta}(t)\psi\big(X^{\eps,\delta}(t), Y^{\eps,\delta}(t), \alpha^\eps(t))dt+\sqrt{\delta}\rho(\alpha^\eps(t))Y^{\eps,\delta}(t)dW_2(t).\end{array}\right.
\end{equation}
Here
\bea \ad \varphi(x, y, i)=a(i)-b(i)x-yh(x, y, i) \ \hbox{
and }\\
\ad \psi(x, y, i)=-c(i)-d(i)y+f(i)xh(x, y, i), \eea
where
$a(\cdot), b(\cdot), c(\cdot), d(\cdot), f(\cdot), \lambda(\cdot), \rho(\cdot)$ are positive functions defined on $\M$,
 $\delta=\delta(\eps)$ depends on $\eps$, $\lim\limits_{\eps\to0}\delta=0$,
$W_1(t)$ and $W_2(t)$ are independent Brownian motions, and
$\alpha^\eps$ is an independent Markov chain with generator $Q/\eps$. As before, the generator $Q$ is assumed to be irreducible so that the Markov chain has a unique stationary distribution given by $(\nu_1,\dots,\nu_{n_0})$.
The function $h(\cdot, \cdot, \cdot)$ is assumed to be positive, bounded, and continuous on $\R^2_+\times\M$.

 For $g(\cdot)= a(\cdot), b(\cdot), c(\cdot), d(\cdot), f(\cdot), \varphi(\cdot), \psi(\cdot)$, define the \textit{averaged quantities} $\bar g:=\sum g(i)\nu_i$, $g_m=\min\{g(i):i\in\M\}, g_M=\max\{g(i):i\in\M\}$. Set
$h_1(x, y):=\sum h(x, y, i)\nu_i$ and $h_2(x, y):=\sum f(i)h(x, y, i)\nu_i.$
The existence and uniqueness of a global positive solution to \eqref{ex1} can be proved in the same manner as in \cite{JJ} or \cite{CDN} and is therefore omitted.
We denote by $Z^{\eps,\delta}_{z, i}(t)=(X^{\eps, \delta}_{z, i}(t), Y^{\eps, \delta}_{z, i}(t))$ the solution to \eqref{ex1} with initial value $\alpha^{\eps}(0)=i\in\M, Z^{\eps,\delta}_{z, i}(0)=z\in\R^{2}_+.$
Consider the averaged equation
\begin{equation}\label{ex2}
\left\{\begin{array}{lll}
\disp {d \over dt}{X}(t)=X(t)\bar\varphi(X(t), Y(t))
=X(t)\left[\bar a-\bar bX(t)-Y(t)h_1(X(t), Y(t))\right]\\
\disp {d \over dt}
Y(t)=Y(t)\bar\psi(X(t), Y(t)))=Y(t)\left[-\bar c-\bar d Y(t)+X(t)h_2(X(t), Y(t))\right]
.\end{array}\right.
\end{equation}
We denote by $\bar Z_z(t)=(\bar X_z(t), \bar Y_z(t))$, the solution to \eqref{ex2} with initial value $\bar Z_z(0)=z$.
\begin{asp}\label{asp5.1}
\begin{enumerate}[(i)]$\text{}$
  \item The system \eqref{ex2} has a finite number of positive
  equilibria and a unique stable limit cycle $\Gamma$. In addition, any positive solution not starting at an equilibrium converges to the stable limit cycle.
  \item The inequality $$\frac{\bar a}{\bar b}h_2\left(\frac{\bar a}{\bar b}, 0\right)>\bar c$$ is satisfied.
\end{enumerate}\end{asp}

\begin{rmk}
  Note that the Jacobian of
$\Big(x\bar\phi(x,y), y\bar\psi(x,y)\Big)^\top$
at $\left(\frac{\bar a}{\bar b}, 0\right)$
has two eigenvalues: $-\bar c+\frac{\bar a}{\bar b}h_2(\frac{\bar a}{\bar b}, 0)$
and $-\frac{\bar b^2}{\bar a}<0$.
If $-\bar c+\frac{\bar a}{\bar b}h_2(\frac{\bar a}{\bar b}, 0)<0$,
then $\left(\frac{\bar a}{\bar b}, 0\right)$ is a stable equilibrium of \eqref{ex2}, which violates condition (i) of Assumption \ref{asp5.1}. This shows that condition (ii) is often contained in condition (i).

We note that the model \eqref{ex2} is quite general and as such conditions on the parameters for the existence and uniqueness of a limit cycle are in general complicated.
\end{rmk}

We can apply Theorem \ref{t:main} to this model if we can verify part (v) of Assumption \ref{asp1} since the other conditions are clearly satisfied.
Since the process $\alpha^\eps(t)$ is ergodic and the diffusion is nondegenerate, an invariant probability measure of the solution $Z^{\eps,\delta}(t)$ is unique if it exists.
It is unlikely that one could find a Lyapunov-type function satisfying the hypothesis of \cite[Theorem 3.26]{YZ} in order to prove the existence of an invariant probability measure.
In addition, the tightness of the family of invariant probability measures $(\mu^{\eps,\delta})_{\eps>0}$ cannot be proved using the methods from \cite{DDT, DDY}.

These difficulties can be overcome
%are resolved
with the help of a new technical tool. We partition the domain $(0,\infty)^2$ into several parts and then construct a truncated Lyapunov-type function. We then estimate the average probability that the solution belongs to a specific part of our partition.
This then allows us to prove that the family of invariant probability measures $(\mu^{\eps,\delta})_{\eps>0}$ is tight on the interior of $\R_+^2$, i.e. for any $\eta>0$, there are  $0<\eps_0,\delta_0<1< L$ such that
for all $\eps<\eps_0,\delta<\delta_0$, the unique invariant probability measure $\mu^{\eps,\delta}$ of $(Z^{\eps,\delta}(t), \alpha^\eps(t))$ satisfies
$$\mu^{\eps,\delta}([ L^{-1}, L]^2)\geq 1-\eta.$$
We are able to prove the following result.

\begin{restatable}{thm}{mainnex}\label{thm5.1}
Suppose Assumption {\rm\ref{asp5.1}} holds. For
%all
sufficiently small $\delta$ and $\eps$, the process given by \eqref{ex1} has a unique invariant probability measure $\mu^{\eps,\delta}$ with support in $\inte\R_+^{2}$ $($where $\inte\R_+^{2}$ denotes the interior of $\R^2_+)$. In addition:
\begin{enumerate}
  \item[a)] If $\lim\limits_{\eps\to0}\dfrac\delta\eps=l\in(0,\infty]$, the family of invariant probability measures $(\mu^{\eps,\delta})_{\eps>0}$ converges weakly to $\mu^0$, the occupation measure of the limit cycle of \eqref{ex2}, as $\eps\to0$ (in the sense
 of Theorem \ref{t:main}).
  \item[b)] If $\lim\limits_{\eps\to0}\dfrac\delta\eps=0$ and at each equilibrium $(x^*, y^*)$ of $(\bar\varphi(x, y), \bar\psi(x, y))$, there is $i^*\in\M$ such that either $\varphi(x^*, y^*, i^*)\ne0$ or $\psi(x^*, y^*, i^*)\ne0$, then the family of invariant probability measures $(\mu^{\eps,\delta})_{\eps>0}$ converges weakly to $\mu^0$, the occupation measure of the limit cycle of \eqref{ex2}, as $\eps\to0$.
\end{enumerate}
\end{restatable}

\begin{rmk}
We note that on any finite time interval $[0,T]$ the solutions to \eqref{ex1} converge to the solutions of \eqref{ex2}. However, in ecology, people are interested in the long term behavior of ecosystems as $T\to \infty$. Therefore, the above result shows rigorously that \eqref{ex2} gives the correct long-term behavior.
\end{rmk}

\subsection{Main ideas
%Sketch
of proof of Theorem \ref{t:main}: A road map}\label{subsec:main}
Because some parts of the proofs are very technical, in order to offer some insight,
%intuition to the reader
we present the main ideas in this subsection.
It aims to provide something like a ``road map'' for the proofs.

Condition (v) of Assumption \ref{asp1} is a tightness assumption for the family of invariant probability measures $(\mu^{\eps,\delta})_{0<\eps<\eps_0}$. This implies that any weak limit of  $(\mu^{\eps,\delta})_{0<\eps<\eps_0}$ is an invariant probability measure of the limit system \eqref{eq2.2}. The main technical issue is to show that any subsequential limit of $(\mu^{\eps,\delta})_{0<\eps<\eps_0}$ does not assign any mass to any of the fixed points of $\bar f$. This is done by a careful analysis of the nature of the deterministic and stochastic systems near the attracting region  $\chi_l:=\{y: \lim_{t\to\infty}\bar X_y(t)=x_l\}$, of an equilibrium $x_l$ of $\bar f$. Note that if $x_l$ is a source then $\chi_l=\{x_l\}$ while if $x_l$ is hyperbolic $\chi_l$ can be an unbounded set. This makes the problem hard.

In Section \ref{sec:3}, using large deviation techniques, we establish the following uniform estimate for the probability that the processes $X_{x,i}^{\eps,\delta}$ and $\bar X_x$ are close on a fixed time interval: For any $R$, $T$,
and $\gamma>0$, there is a $\kappa=\kappa(R,\gamma, T)>0$ such that
\begin{equation}\label{sk1}
\PP\left\{\left|X^{\eps,\delta}_{x,i}(t)-\bar X_x(t)\right|\geq\gamma ~\text{for some}~t\in [0,T] \right\}<\exp\left(-\frac{\kappa}{\eps+\delta}\right),  x\in B_R.
\end{equation}

The main task is to estimate the time of exiting the attracting region, $\chi_l\cap B_R$, of an equilibrium $x_l$. To be precise, we  show that $X_{x,i}^{\eps,\delta}$ leaves small neighborhoods of $\chi_l\cap B_R$ with strictly positive probability in finite time if we start close to  $\chi_l\cap B_R$. We find uniform lower bounds for these probabilities.

 In fact, for any
 %that is
 sufficiently small  $\Delta>0$
 and sufficiently large
 %any
 $R>0$
 %that is large enough
 to include all the sets $M_i, i=1,\dots,n_0$, we can find $\theta_1,\theta_3>0$, $H^\Delta_l>0$, and $ \eps_{l}(\Delta)$ such that for $\eps<\eps_{l}(\Delta)$,
\begin{equation}\label{sk2}
 \PP\left\{\wdt \tau^{\eps,\delta}_{x, i}\leq H^\Delta_l\right\}\geq\psi^{\Delta,\eps}:=\exp\left(-\dfrac{\Delta}{\eps+\delta}\right),\,|x-x_l|<\theta_1,
\end{equation}
where
$$\wdt \tau^{\eps,\delta}_{x, i}:=\inf\{t\geq 0: X^{\eps,\delta}_{x,i}(t)\in B_R \text{ and } \dist(X^{\eps,\delta}_{x,i}(t),\chi_l)\geq \theta_3\}.$$
We prove the estimate \eqref{sk2} in the different cases as follows:
\begin{enumerate}
\item [1)] Suppose that there is an $i^*\in\M$ satisfying
%such that
$\beta'f(x_l,i^*)\neq 0$, where $\beta$ is a normal unit vector of the
stable manifold of \eqref{eq2.2} at $x_l$. Then we estimate the time $\alpha^\eps(t)$ stays in $i^*$ and consider the diffusion in this fixed state, that is
\[
dZ^\delta(t)=f(Z^\delta(t), i^*)dt+\sqrt{\delta}\sigma(Z^\delta(t), i^*)dW(t).
\]
      Since the drift $f(x,i^*)$ is nonzero and pushes us away from the stable manifold of $x^*$, and the diffusion term is small, we can estimate the exit time $\wdt\tau^{\eps,\delta}_{x, i}$.
\item [2)] Suppose $\lim_{\eps\to0}\frac{\delta}{\eps}\in(0,\infty]$ and there is an $i^*$ such that $\beta'\sigma(x_l,i^*)\neq 0$. If $\lim_{\eps\to0}\frac{\delta}{\eps}<\infty$, suppose in addition that $\beta'f(x_l,i)=0, i\in \M$. We estimate the time $\alpha^\eps(t)$ to be
%is
in $i^*$ and consider the diffusion component in the direction $\beta$ in this fixed state
\[
dZ^{\eps,\delta} = \sqrt{\delta}\beta'\sigma(Z^{\eps,\delta},\alpha^\eps)dW(t)
\]
Since the diffusion coefficient does not vanish close to $x_l$, we can do time change so that we get a Brownian motion. Then we can estimate the probability that the exit time exceeds a given number. Ultimately, we show that $Z^{\eps,\delta}$ and $\beta'X^{\eps,\delta}$ are close to each other.
\end{enumerate}
Comparing the rates in \eqref{sk1} with \eqref{sk2} is key to prove the main result  in Section \ref{sec:4}  (see e.g. \cite{CH, kifer12}). The idea is to estimate the time of exiting the attracting region, $\chi_l\cap B_R$, of an equilibrium $x_l$ as well as the time of coming back to this region. Then we prove that eventually, the probability of entering $\chi_l\cap B_R$ is very small compared to the probability of exiting the region.

If we start with $\bar X(0)$ close to $\chi_l\cap B_R$ then after a finite time $\bar X$ will be close to one of the equilibrium points or the limit cycle. Using this together with \eqref{sk1} and \eqref{sk2} we get that there exist neighborhoods $N_1, G_1$ of $\chi_l\cap B_R$ with $N_1\subset G_1$ such that
\[
\PP\{\tau^{\eps,\delta}_{x,i}<L\}>\dfrac1{8}\psi^{\Delta,\eps}, x\in N_1
\]
for some constant $L>0$ and
$$\tau^{\eps,\delta}_{x,i}=\inf\{t\geq0: X_{x,i}^{\eps,\delta}(t)\in B_R\setminus G_1\}.$$
This can be leveraged into showing that with high probability, if we start in $N_1$, we will leave the region $G_1\supset N_1$ in a finite, uniformly bounded, time:
\begin{equation}\label{sk3}
\PP\left\{\tau^{\eps,\delta}_{x,i}<T^{\eps,\delta}_{\Delta,1}\right\}>\frac12, x\in N_1
\end{equation}
where $T^{\eps,\delta}_{\Delta,1}:=C\exp\left(\dfrac{\Delta}{\eps+\delta}\right)$.
Using \eqref{sk1} we can find a constant $\hat T>0$, independent of $\eps$ such that
\begin{equation}\label{sk4}
\PP\left\{X^{\eps, \delta}_{x, i}(\hat T)\notin G_1 \right\}\geq1-\exp\Big(-\dfrac{\kappa}{\eps+\delta}\Big), x\in B_R\setminus N_1
\end{equation}
and that
\begin{equation}\label{sk5}
\PP\left\{X^{\eps, \delta}_{x, i}(t)\notin N_1, \ ~\text{for all}~ t\in[0,\hat T]\right\}\geq1-\exp\Big(-\dfrac{\kappa}{\eps+\delta}\Big), x\in B_R\setminus G_1.
\end{equation}

Note that $T^{\eps,\delta}_{\Delta,1}\to\infty$ as $\eps\to 0$. However, if we pick $\Delta<\kappa/2$, we have
\begin{equation}\label{sk6}
\lim_{\eps\to 0}T_{\Delta,1}^{\eps,\delta}\exp\left(-\dfrac{\kappa}{\eps+\delta}\right)= \lim_{\eps\to 0}\exp\left(\dfrac{\Delta}{\eps+\delta}\right)\exp\left(-\dfrac{\kappa}{\eps+\delta}\right)= 0.
\end{equation}
The estimate \eqref{sk6} shows the exit time is not long compared to the good rate of large deviations,
which will be used to show that invariant probability measures cannot put much mass on the equilibria.
Let $\widetilde X^{\eps, \delta}(t)$ be the stationary solution, whose distribution is $\mu^{\eps,\delta}$ for every time $t\geq 0$. Let  $\tau^{\eps,\delta}$
 be the first exit time of $\widetilde X^{\eps, \delta}(t)$ from $G_1$.
We can show that for any $\eta>0$ we can find $R>0$ such that $\mu^{\eps,\delta}(N_1)\leq 2\eta$ by using \eqref{sk4}, \eqref{sk5}, and \eqref{sk6} to find the probabilities of the events
\begin{align*}
K_1^{\eps, \delta}&=\Big\{\widetilde X^{\eps, \delta}(T_{\Delta,1}^{\eps,\delta})\in N_1, \tau^{\eps,\delta}\geq T_{\Delta,1}^{\eps,\delta}, \widetilde X^{\eps,\delta}(0)\in N_1\Big\}\\
K_2^{\eps, \delta}&=\Big\{\widetilde X^{\eps, \delta}(T_{\Delta,1}^{\eps,\delta})\in N_1, \tau^{\eps,\delta}< T_{\Delta,1}^{\eps,\delta}, \widetilde X^{\eps,\delta}(0)\in N_1\Big\}\\
K_3^{\eps, \delta}&=\Big\{\widetilde X^{\eps, \delta}(T_{\Delta,1}^{\eps,\delta})\in N_1, \widetilde X^{\eps,\delta}(0)\in B_R\setminus N_1\Big\}\\
K_4^{\eps, \delta}&=\Big\{\widetilde X^{\eps, \delta}(T_{\Delta,1}^{\eps,\delta})\in N_1, \widetilde X^{\eps,\delta}(0)\notin B_R\Big\}.
\end{align*}

Similar arguments show that for any $\eta>0$, we can find $R>0$ and neighborhoods $N_1,\dots, N_{n_0-1}$ of $\chi_1\cap B_R,\dots, \chi_{n_0-1}\cap B_R$ such that
$$\limsup_{\eps\to0}  \mu^{\eps,\delta}(\cup_{j=1}^{n_0-1}N_j)\leq 2^{n_0}\eta.$$
Using this fact together with Assumption \ref{asp1} and Lemma \ref{lm2.2} we can establish, by a straightforward modification of the proof of \cite[Theorem 1]{CH}, that for any $\eta>0$ there is neighborhood $N$ of the limit cycle $\Gamma$ such that
$$\liminf_{\eps\to  0} \mu^{\eps,\delta} (N)>1-2^{n_0}\eta.$$

\section{Estimates for the first exit times}\label{sec:3}
Define for any $i=1,\dots,n_0$ and $\theta>0$, the sets
$\chi_i:=\{y: \lim_{t\to\infty}\dist(\bar X_y(t),M_i)=0\}$ and
$M_{i,\theta}:=\{y: \dist(y, M_i)<\theta\}$.
Let $R_0>1$ be large enough such that $B_{R_0-1}$ contains all $M_i$, $i=1,\dots, n_0$.
Fix $\theta_0\in(0,1)$ such that $\{M_{i,2\theta_0}, i=1,\dots, n_0\}$ are mutually disjoint
and $M_{i,2\theta_0}\cap \chi_j=\emptyset$ for  $j<i$. For any $\eta>0$, let $R=R_\eta>0$ such that $\mu^{\eps,\delta}(B_R)>1-\eta$ and $R<R_0$.

The following is a well-known exponential martingale inequality (see \cite[Theorem 1.7.4]{XM}).
\begin{lm}\label{l:exp}(Exponential martingale inequality)
Suppose $(g(t))$ is a real-valued $\F_t$-adapted process and $\int_0^Tg^2(t)dt<\infty$ almost surely. Then for any $a,b>0$ one has
$$\PP\left\{\sup_{t\in[0, T]}\left[\int_0^tg(s)dW(s)-\dfrac{a}{2}\int_0^tg^2(s)ds\right]>b\right\}\leq e^{-ab}.$$
\end{lm}
We will make use of this lemma repeatedly in the proofs to follow. The next result gives us estimates on how close the solutions to \eqref{eq2.1} and \eqref{eq2.3} are on a finite time interval if they have the same starting points.
%Because
The argument of the
proof is pretty standard. For completeness, it  relegated
%appears
to Appendix \ref{a:1}.

\begin{lm}\label{lm2.2}
For any $R$, $T$, and $\gamma>0$, there is a $\kappa=\kappa(R,\gamma, T)>0$ such that
$$\PP\left\{\left|X^{\eps,\delta}_{x,i}(t)-\bar X_x(t)\right|\geq\gamma ~\text{for some}~t\in [0,T] \right\}<\exp\left(-\frac{\kappa}{\eps+\delta}\right),  x\in B_R.$$
\end{lm}
\begin{lm}\label{lm3.1}
Let $N$ be an open set in $\R^d$ and let $\check\tau_{x, i}^{\eps, \delta}$ be any stopping time.
Suppose that there is an $\ell>0$ such that for all starting points
$(x, i)\in N\times\M$ one has $\PP\{\check\tau_{x, i}^{\eps, \delta}<\ell\}\geq a^{\eps,\delta}>0$, where $\lim\limits_{\eps\to0}a^{\eps,\delta}=0$.
Then $\PP\left\{\check\tau_{x, i}^{\eps, \delta}<\dfrac{\ell}{a^{\eps,\delta}}\right\}>1/2$ for $(x, i)\in N\times\M$ if $\eps$ is sufficiently small.
\end{lm}
\begin{lm}\label{lm2.5} The following properties hold:
\begin{enumerate}
\item For any $\theta>0, R>0$, there exists $\wdt T_1>0$ such that
for any $y\in B_R$,
$\bar X_y(t)\in M_{k,\theta}$ for some $t<\wdt T_1$, and some $k\in\{1,\dots, n_0\}$.
\item For any $y\in B_R\setminus\chi_1$ and any $\theta>0$, there exists $\wdt t_y>0$ such that $\bar X_y(t_y)\in\bigcup_{k=2}^{n_0} M_{k,\theta}$.
\item For any $\theta_1>0$, $R\geq R_0$, there exists $\theta_2>0$ such that
$\dist(\bar X_y(t), \chi_1)>\theta_2$ for any $t>0$ and $y\in B_R$ satisfying
$\dist(y,\chi_1)>\theta_1$.
\item Let $\beta$ be a normal unit vector of the stable manifold of \eqref{eq2.2} at an equilibrium $x_l$.
Then for any $m>0$, we can find $\wdt\theta_0>0$ such that
$\{y: |\beta' y|\geq\theta, |y|\leq m\theta\}\cap\chi_l=\emptyset$
for any $\theta\in(0,\wdt\theta_0]$
\end{enumerate}
\end{lm}

The following lemmas show that the process leaves small neighborhoods around the equilibrium points with strictly positive probability in finite time if we start close to the equilibrium points. Furthermore, this probability can be bounded below uniformly for all starting points close to the equilibrium. We need this because we want to show the convergence of the process to the limit cycle $\Gamma$.

\begin{lm}\label{lm3.2}

Consider an equilibrium $x_l$ and suppose there exists $i^*\in\M$ such that $\beta' f(x_l, i^*)\ne 0$ where $\beta$ is a normal unit vector
of the stable manifold of \eqref{eq2.2} at $x_l$.
Then for any $\Delta>0$ that is sufficiently small and any $R>R_0$, we can find $\theta_1,\theta_3>0$, $H^\Delta_l>0$, and $ \eps_{l}(\Delta)$ such that for $\eps<\eps_{l}(\Delta)$,
$$\PP\left\{\wdt \tau^{\eps,\delta}_{x, i}\leq H^\Delta_l\right\}\geq\psi^{\Delta,\eps}:=\exp\left(-\dfrac{\Delta}{\eps}\right),\,x\in M_{l,\theta_1},$$
where
$$\wdt \tau^{\eps,\delta}_{x, i}:=\inf\{t\geq 0: X^{\eps,\delta}_{x,i}(t)\in B_R \text{ and } \dist(X^{\eps,\delta}_{x,i}(t),\chi_l)\geq \theta_3\}.$$
\end{lm}

\begin{proof}
Suppose without loss of generality that $x_l=0$.
Let $\beta$ be a normal vector of the stable manifold at $0$ such that $|\beta|=1$ and $\beta'f(0, i^*)>0$. Since $f$ is locally Lipschitz we can find $a_1>0$ such that
\begin{equation}\label{e:a}
\beta'f(x, i^*)>a_1>0, |x|<\theta_0.
\end{equation}
Then $A_1:=\sup_{x<\theta_0}\left\{\frac{|f(x,i^*)|}{\beta'f(x,i^*)}\right\}<\infty$.

Since $\beta$ is perpendicular to the tangent of the stable manifold at $0$,
we can find $\theta_2\in\left(0, \frac1{2+3A_1}\left(\frac{a_1\Delta}{4|q_{i^*i^*}|}\wedge\theta_0\right)\right)$ such that
\begin{equation}\label{e:dist}
\dist(L_l^{\theta_2}, \chi_l):=\theta_3>0
\end{equation}
where
\begin{equation}\label{defL1}
L_l^{\theta_2}=\{x: |x|\leq(2+3A_1)\theta_2 \text{ and } |\beta'x|>\theta_2\}.
\end{equation}
The continuous dependence of the solutions of \eqref{eq2.2} on the starting point and the fact that $0$ is an equilibrium of \eqref{eq2.2} imply that $\bar X$ stays close to $0$ for a finite time if the starting point is close enough to $0$.
Using this, we can derive from Lemma \ref{lm2.2} that there exist numbers $\theta_1\in(0,\theta_2)$ and $k>0$ such that
\begin{equation}\label{lm3.2-e1}
\PP\left\{| X_{x,i}^{\eps,\delta}(t)|<\theta_2, 0<t<1+\frac{1}{|q_{i^*i^*}|}\right\}>1-\exp\left(-\frac{k}{\eps+\delta}\right)\, \text{ for all } x\in M_{l,\theta_1}, i\in\M.
\end{equation}

 First, we consider the case $\alpha^\eps(0)=i^*$. Because of the independence
  of $\alpha^\eps(\cdot)$ and $W(\cdot)$, if $\alpha^\eps(t)=i^*$ for all  $t\in\Big[0, \frac{\Delta}{|q_{i^*i^*}|}\Big]$, the process $X^{\eps, \delta}_{x, i^*}(\cdot)$ has the same distribution on the time interval $\Big[0, \frac{\Delta}{|q_{i^*i^*}|}\Big]$ as that of
  $Z^{\delta}_x$ given by
\begin{equation}\label{e:Z}
dZ^\delta(t)=f(Z^\delta(t), i^*)dt+\sqrt{\delta}\sigma(Z^\delta(t), i^*)dW(t).
\end{equation}
Define the bounded stopping time $$\rho^{\eps,\delta}_x:=\dfrac{\Delta}{|q_{i^*i^*}|}\wedge \inf\{t>0: |Z_x^\delta(t)|\geq \theta_0\}\wedge \inf\{t>0: \beta'Z_x^\delta(t)\geq \theta_2\}.$$

We have
\begin{equation}\label{e:betaZ}
\beta'Z_x^\delta(\rho_x^{\eps,\delta})=\beta'x+\int_0^{\rho^{\eps,\delta}_x}\beta'f(Z_x^\delta(s), i^*)ds+\int_0^{\rho^{\eps,\delta}_x}\sqrt{\delta}\beta'\sigma(Z_x^\delta(s), i^*)dW(s), |x|\leq \theta_0.
\end{equation}
By the exponential martingale inequality from Lemma \ref{l:exp}, there exists a constant $m_3>0$ independent of $\delta$ such that
$$\PP\left(\Omega^{\eps,\delta,1}_{x}\right)>\frac34 \text{ and }\PP\left(\Omega^{\eps,\delta,2}_{x,i}\right)>\frac34$$
where
\bea \ad
\Omega^{\eps,\delta,1}_{x}:=\Bigg\{-\int_0^t\sqrt{\delta}\beta'\sigma(Z_x^\delta(s), i^*)dW(s)
\\ \aad \quad\qquad\qquad
-\dfrac1{\sqrt{\delta}}\int_0^t\delta\beta'\sigma(Z_x^\delta(s), i^*)\sigma(Z_x^\delta(s), i^*)'\beta ds<m_3\sqrt{\delta}, t\in\left[0,{\rho^{\eps,\delta}_x}\right]\Bigg\}
\eea
and
\bea \ad
\Omega^{\eps,\delta,2}_{x}:=\Bigg\{\left|\int_0^t\sqrt{\delta}\sigma(Z_x^\delta(s), i^*)dW(s)\right|
\\ \aad  \quad\qquad\qquad
-\dfrac1{\sqrt{\delta}}\int_0^t\delta\left|\sigma(Z_x^\delta(s), i^*)\sigma(Z_x^\delta(s), i^*)'\right| ds<m_3\sqrt{\delta}, t\in\left[0,{\rho^{\eps,\delta}_x}\right]\Bigg\}.
\eea
This implies that
\begin{equation}\label{e:12}
\PP\left(\Omega^{\eps,\delta,1}_{x}\cap \Omega^{\eps,\delta,2}_{x,i}\right)>\frac{1}{2}.
\end{equation}

Using \eqref{e:a} and \eqref{e:betaZ} we note that on the set $\Omega^{\eps,\delta,1}_{x}$
\begin{equation}\label{lm3.2-e7}
\begin{aligned}
\beta'Z_x^\delta(\rho_x^{\eps,\delta})>&\beta'x+\int_0^{\rho^{\eps,\delta}_x}\beta'f(Z_x^\delta(s), i^*)ds\\
& -\dfrac1{\sqrt{\delta}}\int_0^{\rho^{\eps,\delta}_x}\beta'\delta\sigma(Z_x^\delta(s), i^*)'\sigma(Z_x^\delta(s), i^*)\beta ds-m_3\sqrt{\delta}\\
\geq&-\theta_2+\int_0^{\rho^{\eps,\delta}_x}a_1ds-m_3\sqrt{\delta}
\end{aligned}
\end{equation}
Let $\delta$ be so small that $m_3\sqrt{\delta}<\dfrac{a_1}2\dfrac{\Delta}{|q_{i^*i^*}|}$.
 If $\rho^{\eps,\delta}_x(\omega)=\dfrac{\Delta}{|q_{i^*i^*}|}$ for some $\omega\in\Omega^{\eps,\delta,1}_{x,i}$, using $\theta_2\leq\dfrac{a_1\Delta}{4|q_{i^*i^*}|}=\dfrac{a_1{\rho^{\eps,\delta}_x}}4$,  we get
$$|\beta'Z_x^\delta({\rho^{\eps,\delta}_x}(\omega))|\leq \theta_2< -\theta_2+a_1{\rho^{\eps,\delta}_x}-m_3\sqrt{\delta},$$
which contradicts \eqref{lm3.2-e7}.
As a result, if $x\leq\theta_2, \omega\in\Omega^{\eps,\delta,1}_{x}$ and $\delta$ is sufficiently small, we have
\begin{equation}\label{e:O1rho}
\rho^{\eps,\delta}_x(\omega)<\dfrac{\Delta}{|q_{i^*i^*}|},
\end{equation}
and
by \eqref{e:betaZ} we have
\begin{equation}\label{lm3.2-e9}
\begin{aligned}
\int_0^{\rho^{\eps,\delta}_x}\beta'f(Z_x^\delta(s), i^*)ds
\leq& |\beta'Z_x^\delta(\rho_x^{\eps,\delta})|+|\beta'x|+\sqrt{\delta}\int_0^{\rho^{\eps,\delta}_x}\left|\sigma(Z_x^\delta(s), i^*)\sigma(Z_x^\delta(s), i^*)'\right| ds+m_3\sqrt{\delta} \\
<&3\theta_2.%, \qquad\text{ if }\, x\leq\theta_1, \omega\in\Omega^{\eps,\delta,1}_{x}
\end{aligned}
\end{equation}
on $\Omega^{\eps,\delta,1}_{x}\cap \Omega^{\eps,\delta,2}_{x}$.
Using \eqref{e:Z} and \eqref{lm3.2-e9}, one sees
that if $\delta$ is sufficiently small and $|x|<\theta_2$ then for $\omega\in \Omega^{\eps,\delta,1}_{x}\cap \Omega^{\eps,\delta,2}_{x}$,
\begin{equation}\label{lm3.2-e8}
\begin{aligned}
|Z_x(\rho^{\eps,\delta}_x)|<&|x|+\int_0^{\rho^{\eps,\delta}_x}|f(Z_x(t),i^*)|dt+\sqrt{\delta}\int_0^{\rho^{\eps,\delta}_x}\left|\sigma(Z_x^\delta(s), i^*)\sigma(Z_x^\delta(s), i^*)'\right| ds+m_3 \sqrt{\delta}\\
<&2\theta_2+A_1 \int_0^{\rho^{\eps,\delta}_x}\beta'f(Z_x(t),i^*)dt\\
<&(2+3A_1)\theta_2<\theta_0,
\end{aligned}
\end{equation}
Combining \eqref{lm3.2-e8} with the definition of $\rho^{\eps,\delta}_x$  shows that
$\beta'Z_x(\rho^{\eps,\delta}_x)=\theta_2$ and
$|Z_x(\rho^{\eps,\delta}_x)|<(2+3A_1)\theta_2$ on $ \Omega^{\eps,\delta,1}_{x}\cap \Omega^{\eps,\delta,2}_{x}$.
As a result of this and \eqref{e:12},
if $|x|\leq\theta_2$,
$$\PP \left\{ \beta 'Z_x(t)\geq\theta_2, |Z_x(t)|\leq(2+3A_1)\theta_2 \text{ for some } t\in\left[0,\frac\Delta{|q_{i^*i^*}|}\right]\right\}\geq \PP\left(\Omega^{\eps,\delta,1}_{x}\cap \Omega^{\eps,\delta,2}_{x,i}\right)>\frac{1}{2}.$$
Let $$\zeta^{\eps, \delta}_{x, i}:=\inf\{t>0: \beta' X^{\eps,\delta}_{x,i}(t)\geq\theta_2, |X^{\eps,\delta}_{x,i}|\leq(2+3A_1)\theta_2\}=\inf\{t>0: X^{\eps,\delta}_{x,i}\in L_l^{\theta_2}\}.$$
Using the independence of $\alpha^\eps$, the paragraph before equation \eqref{e:Z}, and the last two equations, we obtain
\begin{equation}\label{lm3.2-e2}
\PP\left\{\zeta^{\eps, \delta}_{x, i^*}\leq \dfrac{\Delta}{|q_{i^*i^*}|}\right\}> \dfrac12\PP\left\{\alpha^\eps_{i^*}(t)=i^*, \ ~\text{for all}~  t\in\left[0, \dfrac{\Delta}{|q_{i^*i^*}|}\right]\right\}=\dfrac12\exp\left(-\dfrac{\Delta}{\eps}\right), \text{ if } |x|\leq\theta_1.
\end{equation}%\label{lm3.2-e1}

Since $\alpha^\eps(t)$ is ergodic, for any sufficiently small $\eps$, i.e., small enough $\Delta$,
\begin{equation}\label{lm3.2-e3}
\PP\{\alpha^\eps_i(t)=i^* \mbox{ for some } t\in [0, 1]\}>{3 \over 4},  i\in\M.
\end{equation}%\label{lm3.2-e1}
 By the strong Markov property, we derive from \eqref{lm3.2-e1}, \eqref{lm3.2-e2}, and \eqref{lm3.2-e3} that for all $(x,i)\in M_{l,\theta_1}\times\M$ and for $\eps$ sufficiently small
\begin{equation}\label{lm3.2-e5}
\PP\left\{\zeta^{\eps, \delta}_{x, i}<1+\dfrac{\Delta}{|q_{i^*i^*}|}\right\}\geq {1\over 4}\exp\left(-\dfrac{\Delta}{\eps}\right).
\end{equation}
The proof is complete by combining this estimate with \eqref{e:dist}.
\end{proof}

\begin{lm}\label{lm3.3}
Suppose that $\lim\limits_{\eps\to0}{\delta\over\eps}=r>0$. Assume that at the equilibrium point $x_l$, one has $f(x_l, i)=0$ for all $i\in\M$, and there is $i^*\in\M$ for which
$\beta'\sigma(x_l,i^*)\ne 0$,
where $\beta$ is a normal unit vector of the stable manifold of \eqref{eq2.2} at $x_l$. Then for any  sufficiently small $\Delta>0$ and any $R>R_0$, we can find $\theta_1,\theta_3>0$, $H^\Delta_l>0$, and $ \eps_{l}(\Delta)>0$ such that for $\eps<\eps_{l}(\Delta)$,
$$\PP\left\{\wdt \tau^{\eps,\delta}_{x, i}\leq H^\Delta_1\right\}\geq\psi^{\Delta,\eps}:=\exp\Big(-\dfrac{\Delta}{\delta}\Big),~\text{for all}~\,(x,i)\in M_{l,\theta_1}\times\M,$$
where
$$\wdt \tau^{\eps,\delta}_{x, i}=\inf\{t\geq 0: X^{\eps,\delta}_{x,i}(t)\in B_R \text{ and } \dist(X^{\eps,\delta}_{x,i}(t),\chi_l)\geq \theta_3\}.$$
\end{lm}

\begin{proof}
We can assume without loss of generality that $x_l=0$ and $\lim\limits_{\eps\to0}{\delta\over\eps}=1$. Since $\sigma$ is locally Lipschitz, we can find $a_2>0$ such that
\begin{equation}\label{e:a2}
a_2<\beta'(\sigma\sigma')(y, i^*)\beta, |y|< \theta_0.
\end{equation}
Let $K_l>0$ be such that $|f(x,i)|<K_l|x|$ and $|(\sigma'\sigma)(x,i)|<K_l$ if $|x|<\theta_0, i\in \M$. Fix $T>0$  such that  $\dfrac{a_2\nu_{i^*}T}{2}>1$ and let $\theta_1>0$ be such that
\begin{equation}\label{e:theta1}
(2+K_lT)^2e^{K_lT}\theta_1<\theta_0%\wedge\Delta
\end{equation}
and
$\dist(L_l^{\theta_1}, \chi_l):=\theta_3>0$ where
\begin{equation}\label{defL1}
L_l^{\theta_1}:=\{x: |x|\leq (2+K_lT)^2e^{K_lT}\theta_1 \text{ and } |\beta'x|>\theta_1\}.
\end{equation}
Define
$$\zeta_{t,x, i}:=\inf\left\{u>0:\int_0^u\beta'(\sigma\sigma')\left(\left(1\wedge\frac{\theta_0}{|X^{\eps, \delta}_{x, i}(s)|}\right)X^{\eps, \delta}_{x, i}(s),\alpha^\eps_i(s)\right)\beta ds\geq t\right\}.$$
For all $t\geq 0$, we have by \eqref{e:a2} and the ergodicity of the Markov chain $\alpha^\eps_i$ that $$\PP(\zeta_{t,x,i}<\infty)=1, |x|<\theta_0.$$  As a result the process $(M(t))_{t\geq 0}$ defined by
 $$M(t)=\int_0^{\zeta_{t,x,i}}\beta'\sigma\left(\left(1\wedge\frac{\theta_0}{|X^{\eps, \delta}_{x, i}(s)|}\right)X^{\eps, \delta}_{x, i}(s),\alpha^\eps_i(s)\right)dW(s)$$ is a Brownian motion. This follows from the fact that $M(t)$ is a continuous martingale with quadratic variation $[M,M]_t=t, t\geq 0$.

Set $\theta_2:=(2+K_lT)\theta_1$. Since $M(1)$ has the distribution of a standard normal, for sufficiently small $\delta$, we have the estimate
\begin{equation}\label{e:Omega2}
\PP\{\sqrt{\delta}M(1)>\theta_2\}\geq\dfrac12\exp\left(-\dfrac{\theta_2^2}{\delta}\right), |x|<\theta_0.
\end{equation}
Using the large deviation principle (see \cite{HYZ}), we
%can
can find $a_3=a_3(T)>0$ such that

\begin{equation}\label{ergodic-alpha}
\PP\left\{\dfrac{1}{T}\int_0^{T}\1_{\{\alpha_i^\eps(s)=i^*\}}ds>\dfrac{\nu_{i^*}}2\right\}\geq1-\exp\left(-\dfrac{a_3}{\eps}\right).
\end{equation}
%Let $\Delta $ be such that \textcolor{red}{is something missing here?} and $\Delta<a_3$.
Equation \eqref{e:a2}, the definition of $\zeta_{t,x,i}$, and $\dfrac{a_2\nu_{i^*}T}{2}>1$ yield
\bea \ad \PP\left\{\int_0^T\beta'(\sigma\sigma')\left(\left(1\wedge\frac{\theta_0}{|X^{\eps, \delta}_{x, i}(s)|}\right)X^{\eps, \delta}_{x, i}(s),\alpha^\eps_i(s)\right)\beta ds
 \geq \dfrac{a_2\nu_{i^*}T}{2} \right\}\geq1-\exp\left(-\dfrac{a_3}{\eps}\right),~|x|<\theta_0,\eea
which leads to
\begin{equation}\label{e:zeta1}
\PP\{\zeta_{1, x, i}\leq T\}\geq1-\exp\left(-\dfrac{a_3}{\eps}\right), |x|<\theta_0.
\end{equation}
Define for $|x|<\theta_0$, $i\in\M$
$$
\begin{aligned}
\Omega_{x,i}^{\eps,\delta,3}:=\bigg\{&\left|\sqrt{\delta}\int_0^{t}\sigma\left(\left(1\wedge\frac{\theta_0}{|X^{\eps, \delta}_{x, i}(s)|}\right)X^{\eps, \delta}_{x, i}(s),\alpha^\eps_i(s)\right)dW(s)\right|\\
&\quad<\dfrac{\theta_2}\delta\int_0^t\delta\left|(\sigma'\sigma)\left(\left(1\wedge\frac{\theta_0}{|X^{\eps, \delta}_{x, i}(s)|}\right)X^{\eps, \delta}_{x, i}(s),\alpha^\eps_i(s)\right)\right|ds+\theta_2\leq (K_lT+1)\theta_2, t\in[0,T]\bigg\}
\end{aligned}
$$
and note that the last inequality holds by the definition of $K_l$.
By Lemma \ref{l:exp}
\begin{equation}\label{e:Omega3}
\PP(\Omega_{x,i}^{\eps,\delta,3})\geq 1-\exp\left(-\frac{2\theta_2^2}{\delta}\right), |x|<\theta_0.
\end{equation}

Define the stopping time
$$\zeta_{x,i}=\inf\{t>0: |\beta'X^{\eps,\delta}_{x,i}(t)|\geq\theta_1\}\wedge \inf\{t>0: |X^{\eps,\delta}_{x,i}(t)|\geq (K_l+2)\theta_2e^{K_lT} \}.$$
If  $|x|\leq\theta_1$ and $\omega\in\{\sqrt{\delta}M(1)>\theta_2\}\cap \{\zeta_{1,x,i}\leq T\}\cap \Omega_{x,i}^{\eps,\delta,3}$,
we claim that we must have
\begin{equation}\label{ezT}
\zeta_{x,i}<T.
\end{equation}
We argue by contradiction.
Suppose the three events $\{\sqrt{\delta}M(1)>\theta_2\}$, $\{\zeta_{1,x,i}\leq T\}$, and $\{\zeta_{x, i}\geq\zeta_{1,x,i}\}$ happen simultaneously. Then we get the contradiction
\begin{align*}
\theta_2=(2+K_lT)\theta_1& <\sqrt{\delta}M(1)=\sqrt{\delta}\int_0^{\zeta_{1,x,i}}\beta'\sigma\left(\left(1\wedge\frac{\theta_0}{|X^{\eps, \delta}_{x, i}(s)|}\right)X^{\eps, \delta}_{x, i}(s),\alpha^\eps_i(s)\right)dW(s)\\
& \leq |\beta'X^{\eps, \delta}_{x, i}(\zeta_1)|+|\beta'x|+\Big|\int_0^{\zeta_{1,x,i}}\beta'f(X^{\eps, \delta}_{x, i}(s),\alpha^\eps_i(s))ds\Big|\\
& \leq2\theta_1+\int_0^{\zeta_{1,x,i}}K_l|\beta'X^{\eps,\delta}_{x,i}(s)|ds< (2+K_lT)\theta_1=\theta_2,
\end{align*}
where we used that $\left(1\wedge\frac{\theta_0}{|X^{\eps, \delta}_{x, i}(s)|}\right)X^{\eps, \delta}_{x, i}(s)=X^{\eps, \delta}_{x, i}(s)$ if $s<\zeta_{x,i}$ by the definition of $\zeta_{x,i}$ and \eqref{e:theta1}.

For  $|x|\leq\theta_1$ and $\omega\in\{\sqrt{\delta}M(1)>\theta_2\}
\cap \{\zeta_{x,i}\leq T\}\cap \Omega_{x,i}^{\eps,\delta,3}$,
for any $0\leq t\leq \zeta_{1,x,i}\leq T$,
$$
\begin{aligned}
|X^{\eps,\delta}_{x,i}(t)|\leq& |x|+ \sqrt{\delta}\left|\int_0^{t}\sigma\big(X^{\eps, \delta}_{x, i}(s),\alpha^\eps_i(s)\big)dW(s)\right|
+\int_0^t|f(X^{\eps, \delta}_{x, i}(s),\alpha^\eps_i(s))|ds\\
<& (K_lT+2)\theta_2+K_l\int_0^t |X^{\eps, \delta}_{x, i}(s)|ds.
\end{aligned}
$$
This together with Gronwall's inequality implies that
$$|X^{\eps,\delta}_{x,i}(t)|< (K_lT+2)\theta_2 e^{K_l T}, t\in [0, \zeta_{x,i}]$$
Thus for  $|x|\leq\theta_1$ and $\omega\in\{\sqrt{\delta}M(1)>\theta_2\}\cap \{\zeta_{x,i}\leq T\}\cap \Omega_{x,i}^{\eps,\delta,3}$,
we have that $\zeta_{x,i}<T$ and
 $X^{\eps,\delta}_{x,i}(\zeta_{x,i})<(K_lT+2)\theta_2 e^{K_l T}$
and $\beta' X^{\eps,\delta}_{x,i}(\zeta_{x,i})\geq\theta_1$.

Since $\theta_2<a_3$ and $\lim_{\eps\to 0}\frac{\delta}{\eps}=1$ we have by \eqref{e:Omega2}, \eqref{e:zeta1}, \eqref{e:Omega3} and \eqref{ezT} that for all sufficiently small $\eps$
$$
\PP(\{\sqrt{\delta}M(1)>\theta_2\}\cap \{\zeta_{x,i}\leq T\}\cap \Omega_{x,i}^{\eps,\delta,3})\geq\dfrac14\exp\left(-\dfrac{\theta_2^2}\delta\right)
\geq \dfrac14\exp\left(-\dfrac{\Delta}\delta\right), |x|<\theta_1
$$
if $\Delta<\theta_2^2$,
which completes the proof.
\end{proof}

\begin{lm}\label{lm3.4}
Suppose that $\lim\limits_{\eps\to0}\dfrac\delta\eps=\infty$.
Assume that at the equilibrium point $x_l$ one can find $i^*\in\M$ such that $\beta'\sigma(x_l,i^*)\ne 0$
where $\beta$ is a normal unit vector of the stable manifold of \eqref{eq2.2} at $x_l$.
Then for any sufficiently small $\Delta>0$ and any $R<R_0$ we can find $\theta_1,\theta_3>0$, $H^\Delta_l>0$,and $ \eps_{1}(\Delta)$ such that for $\eps<\eps_{1}(\Delta)$,
$$\PP\left\{\wdt \tau^{\eps,\delta}_{x, i}\leq H^\Delta_l\right\}\geq\psi^{\Delta,\eps}:=\exp\Big(-\dfrac{\Delta}{\delta}\Big)~\text{for all}~\,(x,i)\in M_{l,\theta_1}\times\M ,$$
where
$$\wdt \tau^{\eps,\delta}_{x, i}=\inf\{t\geq 0: X^{\eps,\delta}_{x,i}(t)\in B_R \text{ and } \dist(X^{\eps,\delta}_{x,i}(t),\chi_l)\geq \theta_3\}.$$
\end{lm}

\begin{proof}
Assume, as in the previous lemmas, that $x_l=0$.
Pick a number $a_2>0$ for which $$a_2<\beta'(\sigma\sigma')(y, i^*)\beta, |y|<\theta_0.$$
Let $K_l>0$ be such that $|\bar f(x)|<K_l|x|$ and $|(\sigma'\sigma)(x,i)|<K_l$ whenever $|x|<\theta_0$,
and fix $T>0$ such that  $\dfrac{a_2\nu_{i^*}T}{2}>1$.
Let $\theta_1>0$ be such that $(3+K_lT)^2e^{K_lT}\theta_1<\theta_0$
and
$\dist(L_l^{\theta_1}, \chi_l):=\theta_3>0$ where
\begin{equation}\label{defL11}
L_l^{\theta_1}=\{x: |x-x_l|\leq (3+K_lT)^2e^{K_lT}\theta_1 \text{ and } |\beta'(x-x_l)|>\theta_1\}.
\end{equation}

Define $\theta_2=(3+K_lT)\theta_1$ and let $a_2, M(t), T, \zeta_{1,x,i}$ be as in the proof of Lemma \ref{lm3.3}.
Arguing as in the proof of \eqref{e:zeta1}, we can find $a_3>0$ such that $$\PP\big\{\zeta_{1,x,i}\leq T\big\}\geq 1-\exp\left(-\dfrac{a_3}\eps\right), |x|<\theta_0.$$
Since $\bar f(0)=0$, we can apply the large deviation principle
(see \cite{HYZ}) to show that
there is $\kappa=\kappa(\Delta)>0$ such that
\begin{equation}\label{e:estA}
\PP(A)\geq 1-\exp\left(-\dfrac{\kappa}\eps\right),
\end{equation}
where $A:=\left\{\left|\int_0^{u}f(0,\alpha_i^\eps(s))ds\right|
<\theta_1,  \text{ for all } u\in[0,T]\right\}$.
The estimates
\bea
M(1) \ad =\int_0^{\zeta_{1,x,i}}\beta'\sigma(X^{\eps,\delta}_{x, i}(s), \alpha_i^\eps(s))dW(s)\\
\ad \leq |\beta'X^{\eps, \delta}_{x, i}(\zeta_{1,x,i})|+|\beta'x|+\Big|\int_0^{\zeta_{1,x,i}}\beta' f(0,\alpha_i^\eps(s))ds\Big|\\
\aad \ +\int_0^{\zeta_{1,x,i}}\big|\beta'\big(f(X^{\eps, \delta}_{x, i}(s),\alpha_i^\eps(s))-f(0,\alpha_i^\eps(s))\big)\big|ds.
\eea
and
$$
\begin{aligned}
|X^{\eps,\delta}_{x,i}(t)|\leq& |x|+ \sqrt{\delta}\left|\int_0^{t}\sigma\big(X^{\eps, \delta}_{x, i}(s),\alpha_i^\eps(s)\big)dW(s)\right|
+\int_0^t|\bar f(X^{\eps, \delta}_{x, i}(s))|ds\\
&+\int_0^t |\bar f(X^{\eps, \delta}_{x, i}(s))-f(X^{\eps, \delta}_{x, i}(s),\alpha_i^\eps(s)|ds
\end{aligned}
$$
together with arguments similar to those from the proof of Lemma \ref{lm3.3} show that
$$
\PP\left\{ X^{\eps,\delta}_{x,i}(t)\in L_l^{\theta_1} \text{ for some } t\in[0,T]\right\}\geq\dfrac14\exp\left(-\frac{\Delta}\delta\right), (x,i)\in M_{l,\theta_1}\times\M
$$
if $\delta$ is sufficiently small.
\end{proof}
\begin{rmk}\label{r:dens}
The results in this section still hold true if one assumes the generator $Q(\cdot)$ of $\alpha(\cdot)$ is state dependent -- see an explanation of the exact setting in Remark \ref{r:state}.  By the large deviation principle in \cite[Section 3]{budhiraja2018large} and the truncation arguments in Lemma \ref{lm2.1}, we can obtain Lemma \ref{lm2.2} for the case of state-dependent switching.
It should be noted that while \cite{budhiraja2018large} only considers Case 1 of \eqref{eq:ep-dl}, using the variational representation,
the arguments in \cite[Section 3]{budhiraja2018large} can be applied to obtain Lemma \ref{lm2.2} for the other cases.

 We can also infer from the large deviation principle that \eqref{lm3.2-e3}, \eqref{ergodic-alpha} and \eqref{e:estA} hold in this setting.
	As a result, Lemmas \ref{lm3.2}, \ref{lm3.3} and \ref{lm3.4} hold. These lemmas, in combination with the proofs from Section \ref{sec:4} imply that the main result, Theorem \ref{t:main},  remains unchanged if one has state-dependent switching.
	
\end{rmk}
\section{Proof of the main result}\label{sec:4}
This section provides the proofs of the convergence of $\mu^{\eps,\delta}$ for the three cases
given in \eqref{eq:ep-dl}.

\begin{prop}\label{p:nbhd}
For every $\eta>0$, there exists $R>R_0$ and neighborhoods $N_1,\dots, N_{n_0-1}$ of $\chi_1\cap B_R,\dots, \chi_{n_0-1}\cap B_R$ such that
$$\limsup_{\eps\to0}  \mu^{\eps,\delta}(\cup_{j=1}^{n_0-1}N_j)\leq 2^{n_0}\eta.$$
\end{prop}

\begin{proof}
For any $\eta>0$, let $R>R_0$ be
such that $\mu^{\eps,\delta}(B_R)\geq 1-\eta.$ Define
$$S_1=\{y\in B_R: \dist(y,\chi_1\cap B_R)< \theta_0\}$$
In view of Lemma \ref{lm2.5}, there exists $c_2>0$ such that for all $t\geq 0$
\begin{equation}\label{extra-e3.1}
\dist(\bar X_y(t), \chi_1)\geq 2c_2\,\text{ for any }\,y\in B_R\setminus S_1.
\end{equation}
Define
$$G_1=\{y\in B_R: \dist(y,\chi_1\cap B_R)< c_2\}.$$
There exists $c_3>0$ such that
\begin{equation}\label{extra-e3.8}
\dist(\bar X_y(t), \chi_1)\geq 2c_3 \text{ for any }y\in B_R\setminus G_1,\,t\geq 0.
\end{equation}
Note that we have $2c_3\leq c_2$ and $2c_2\leq\theta_0$.
Define
$$N_1=\{y\in B_R: \dist(y,\chi_1\cap B_R)< c_3\}$$
In view of Lemma \ref{lm2.5},
for any $y\notin\chi_1$, there exists $\wdt t_y$ such that  $\bar X_y(\wdt t_y)\in M_{i,\theta_0}\cap (B_R\setminus S_1)$
for some $i>1$.
This fact together with
the continuous dependence of solutions to initial values
and \eqref{extra-e3.1} implies that
there exists $\hat T>0$ such that
\begin{equation}\label{extra-e3.3}
\dist(\bar X_y(t), \chi_1)\geq 2c_2\,\text{ for any }\,t\geq \hat T, y\in B_R\setminus N_1.
\end{equation}
Let $\kappa=\kappa(R,c_3, \hat T)$ be as in Lemma \ref{lm2.2} and $\Delta<\frac{\kappa}2$ and $\theta_1$ and $\psi^\Delta_\eps$ be as in one of the Lemmas \ref{lm3.2}, \ref{lm3.3} and \ref{lm3.4} (depending on which case we are considering).
We have
\begin{equation}\label{extra-e3.6}
\PP(\wdt \tau_{x,i}^{\eps,\delta}<H^\Delta)\geq\psi^\Delta_\eps, x\in M_{1,\theta_1}
\end{equation}
where, as in Section \ref{sec:3}, the stopping time is
$$\wdt \tau^{\eps,\delta}_{x, i}=\inf\{t\geq 0: X^{\eps,\delta}_{x,i}(t)\in B_R \text{ and } \dist(X^{\eps,\delta}_{x,i}(t),\chi_1)\geq \theta_3\}.$$
Define
$$\tau^{\eps,\delta}_{x,i}=\inf\{t\geq0: X_{x,i}^{\eps,\delta}(t)\in B_R\setminus G_1\}.$$
It follows from part (1) of Lemma \ref{lm2.5} that for any $x\in N_1$, there exists a $\wdt T_1>0$ such that
$\bar X_x(t_x)\in \bigcup_{j=1}^{n_0} M_{j,\frac{\theta_1}2}\text{ for some } t_x\leq \wdt T_1$.

Suppose $\bar X_x(t_x)\in \bigcup_{j=2}^{n_0}M_{j,\frac{\theta_1}2}$.
Note that  $\bigcup_{j=2}^{n_0} M_{j,\frac{\theta_1}2} \cap M_{1,c_3}=\emptyset$, $\theta_1<\theta_0$ and that by construction,  $M_{1,2\theta_0}\cap \chi_j=\emptyset, j>1$. These facts imply that $\bigcup_{j=2}^{n_0} M_{j,\frac{\theta_1}2}\cap N_1=\emptyset$.
This together with Lemma \ref{lm2.2} and \eqref{extra-e3.3} implies
\begin{equation}\label{ex.e1}
\PP\{\tau^{\eps,\delta}_{x,i}<\tilde T_1+\hat T\}>\frac12
\end{equation}
for small $\eps>0$.

When $\eps$ is sufficiently small, we have by Lemma \ref{lm2.2} (applied with $\gamma=\frac{\theta_1}{2})$ that for any $x\in N_1$ satisfying
 $\bar X_x(t_x)\in M_{1,\frac{\theta_1}2}$ that
\begin{equation}\label{extra-e3.4}
\PP\{X^{\eps,\delta}_{x,i}(t_x)\in M_{1,\theta_1}\}>\frac12.
\end{equation}
Similarly to \eqref{extra-e3.3},
there exists a $\wdt T_2>0$ such that
$$
\dist(\bar X_y(t), \chi_1)\geq 2c_2\,\text{ for any }\,t\geq \wdt T_2, y\in B_R, \dist(y,\chi_1)\geq\theta_3,
$$
which implies that by Lemma \ref{lm2.2}, for sufficiently small $\eps>0$,
\begin{equation}\label{extra-e3.5}
\PP\left\{\dist(X^{\eps,\delta}_{x,i}(\wdt T_2), \chi_1)\geq c_2\right\}>\frac{1}{2}\,\text{ for any }\, x\in B_R, \dist(x,\chi_1)\geq\theta_3, i\in\M.
\end{equation}

Putting together \eqref{extra-e3.6}, \eqref{extra-e3.4}, and \eqref{extra-e3.5}
we deduce that
\begin{equation}\label{ex.e2}
\PP\{\tau^{\eps,\delta}_{x,i}<\wdt T_1 +H^\Delta+\wdt T_2\}>\dfrac1{4}\psi^\Delta_{\eps}.
\end{equation}
for $\eps$ sufficiently small.
Combining \eqref{ex.e1} and \eqref{ex.e2},
we get that
\begin{equation}\label{ex.e22}
\PP\{\tau^{\eps,\delta}_{x,i}<H^\Delta+\wdt T_1+\wdt T_2+\hat T\}>\dfrac1{8}\psi^\Delta_{\eps}, x\in N_1.
\end{equation}
Define $T^{\eps,\delta}_{\Delta,1}:=4\dfrac{H^\Delta+\wdt T_1+\wdt T_2+\hat T}{\psi_{\eps,\delta}^\Delta}$. Applying Lemma \ref{lm3.1} to \eqref{ex.e22},
we have
\begin{equation}\label{ex.e3}
\PP\left\{\tau^{\eps,\delta}_{x,i}<T^{\eps,\delta}_{\Delta,1}\right\}>\frac12, x\in N_1.
\end{equation}
We will argue by contradiction that $\limsup\limits_{\eps\to0}\mu^{\eps,\delta}(N_1)\leq 2\eta$.
 Assume that $\limsup\limits_{\eps\to0}\mu^{\eps,\delta}(N_1)>2\eta>0$.
Since $\Delta<\kappa/2$, we have
\begin{equation}\label{e:TD}
\lim_{\eps\to 0}T_{\Delta,1}^{\eps,\delta}\exp\left(-\dfrac{\kappa}{\eps+\delta}\right)= 0.
\end{equation}
Let $\widetilde X^{\eps, \delta}(t)$ be the stationary solution, whose
distribution is $\mu^{\eps,\delta}$
for every time $t\geq 0$. Let  $\tau^{\eps,\delta}$
 be the first exit time of $\widetilde X^{\eps, \delta}(t)$ from $G_1$.
Define the events
\begin{align*}
K_1^{\eps, \delta}&=\Big\{\widetilde X^{\eps, \delta}(T_{\Delta,1}^{\eps,\delta})\in N_1, \tau^{\eps,\delta}\geq T_{\Delta,1}^{\eps,\delta}, \widetilde X^{\eps,\delta}(0)\in N_1\Big\}\\
K_2^{\eps, \delta}&=\Big\{\widetilde X^{\eps, \delta}(T_{\Delta,1}^{\eps,\delta})\in N_1, \tau^{\eps,\delta}< T_{\Delta,1}^{\eps,\delta}, \widetilde X^{\eps,\delta}(0)\in N_1\Big\}\\
K_3^{\eps, \delta}&=\Big\{\widetilde X^{\eps, \delta}(T_{\Delta,1}^{\eps,\delta})\in N_1, \widetilde X^{\eps,\delta}(0)\in B_R\setminus N_1\Big\}\\
K_4^{\eps, \delta}&=\Big\{\widetilde X^{\eps, \delta}(T_{\Delta,1}^{\eps,\delta})\in N_1, \widetilde X^{\eps,\delta}(0)\notin B_R\Big\}.
\end{align*}
Note that the above events are disjoint and have union $N_1$. As such $$\mu^{\eps,\delta}(N_1)=\sum_{n=1}^4\PP\{K_n^{\eps, \delta}\}.$$
Using \eqref{ex.e3},  we get that
\begin{equation}\label{e:K1K4}
\PP(K_1^{\eps, \delta})\leq \dfrac12\mu^{\eps,\delta}(N_1) \ \hbox{ and  }
\ \PP(K_4^{\eps, \delta})\leq 1-\mu^{\eps,\delta}(B_R)< \eta.
\end{equation}
Next, we estimate $\PP(K_3^{\eps, \delta}).$
It follows from Lemma \ref{lm2.2}, \eqref{extra-e3.8},
 and \eqref{extra-e3.3} that if $\eps$ is sufficiently small
then
$$\PP\left\{X^{\eps, \delta}_{x, i}(\hat T)\notin G_1 \right\}\geq1-\exp\Big(-\dfrac{\kappa}{\eps+\delta}\Big), x\in B_R\setminus N_1$$
and
$$\PP\left\{X^{\eps, \delta}_{x, i}(t)\notin N_1, \ ~\text{for all}~ t\in[0,\hat T]\right\}\geq1-\exp\Big(-\dfrac{\kappa}{\eps+\delta}\Big), x\in B_R\setminus G_1.$$
Using the last two estimates together with the Markov property one sees that for any $x\in B_R\setminus G_1, i\in\M, s\in[0,T_{\Delta,1}^{\eps,\delta}]$,
\begin{equation}\label{e4.3}
\begin{aligned}
\PP\Big\{ X_{x,i}^{\eps, \delta}&(s )\in N_1\Big\}\\
=&\PP\Big\{ X_{x,i}^{\eps, \delta}(s )\in N_1,  X_{x,i}^{\eps,\delta}(\hat T)\notin B_R\setminus G_1 \Big\}\\
&+\sum_{n=2}^{\lf s /\hat T\rf }\PP\Big\{ X_{x,i}^{\eps, \delta}(s )\in N_1,  X_{x,i}^{\eps,\delta}(n\hat T)\notin B_R\setminus G_1 ,  X_{x,i}^{\eps,\delta}(\iota\hat T)\in B_R\setminus G_1 , \iota=1,...,n-1\Big\}\\
&+\PP\Big\{ X_{x,i}^{\eps, \delta}(s )\in N_1,  X_{x,i}^{\eps,\delta}(\iota\hat T)\in B_R\setminus G_1 , \iota=1,...,[s /\hat T]\Big\}\\
\leq& \PP\Big\{ X_{x,i}^{\eps,\delta}(\hat T)\notin B_R\setminus G_1\Big\}+\sum_{n=2}^{\lf s /\hat T\rf}\PP\Big\{ X_{x,i}^{\eps,\delta}(n\hat T)\notin B_R\setminus G_1 ,  X_{x,i}^{\eps,\delta}((n-1)\hat T)\in B_R\setminus G_1 \}
\\
&+\PP\left\{X^{\eps, \delta}_{x, i}(t)\in N_1, \,\text{ for some } t\in \left[\left\lf s /\hat T\right\rf\hat T, \left\lf s /\hat T\right\rf\hat T+\hat T\right], X^{\eps, \delta}_{x, i}\left(\left\lf s /\hat T\right\rf\hat T\right)\in B_R\setminus G_1 \right\}\\
\leq&\left(\left\lf s /\hat T\right\rf+1\right)\exp\left(-\dfrac{\kappa}{\eps+\delta}\right)\\
\leq&\left(s /\hat T+1\right)\exp\left(-\dfrac{\kappa}{\eps+\delta}\right),
\end{aligned}
\end{equation}
where $\lf s /\hat T\rf $ denotes the integer part of $s /\hat T$.

Note that similar arguments show that \eqref{e4.3} also holds for all $s\in[\hat T,T^{\eps,\delta}_{x,i}]$ and $x\in B_R \setminus N_1$.
It  follows from this with $s=T_{\Delta,1}^{\eps,\delta}$,
%that
$$\PP(K_3^{\eps, \delta})=\PP\left\{\widetilde X^{\eps, \delta}(T_{\Delta,1}^{\eps,\delta})\in N_1, \widetilde X^{\eps,\delta}(0)\in B_R\setminus N_1\right\}\leq\left(T_{\Delta,1}^{\eps,\delta}/\hat T+1\right)\exp\left(-\dfrac{\kappa}{\eps+\delta}\right).$$
This together with \eqref{e:TD} implies that
\begin{equation}\label{e:K3}
\lim_{\eps\to  0}\PP(K_3^{\eps,\delta})=0.
\end{equation}
Using \eqref{e4.3} and the strong Markov property, we get
\begin{equation}\label{e:K2}
\begin{aligned}
\PP(K_2^{\eps, \delta})=&\PP\Big\{\widetilde X^{\eps, \delta}(T_{\Delta,1}^{\eps,\delta})\in N_1, \tau^{\eps,\delta}< T_{\Delta,1}^{\eps,\delta}, \widetilde X^{\eps,\delta}(0)\in N_1\Big\}\\
=&\int_0^{T_{\Delta,1}^{\eps,\delta}}\PP\{\tau^{\eps,\delta}\in dt\}\left[\sum_{i\in\M}\int_{\partial G_1}\PP\left\{ X_{x,i}^{\eps, \delta}(T_{\Delta,1}^{\eps,\delta}-t)\in N_1\right\}\PP\left\{\alpha^\eps(t)=i, \widetilde X^{\eps, \delta}(t)\in dx\right\}\right]\\
\leq& \left(T_{\Delta,1}^{\eps,\delta}/\hat T+1\right)\exp\left(-\dfrac{\kappa}{\eps+\delta}\right)\\
\to&0\text{ as } \eps\to0 \,\text{ due to }\,\eqref{e:TD}.
\end{aligned}
\end{equation}
Putting together the estimates \eqref{e:K1K4}, \eqref{e:K2}, and \eqref{e:K3}, we see that $$\limsup\limits_{\eps\to0}\mu^{\eps,\delta}(N_1)\leq \dfrac12\limsup\limits_{\eps\to0}\mu^{\eps,\delta}(N_1)+0+0+\eta,$$
which contradicts the assumption that $\limsup\limits_{\eps\to0}\mu^{\eps,\delta}(N_1)>2\eta$. We have therefore shown that
\[
\lim_{\eps\to 0} \mu^{\eps,\delta}(N_1)\leq 2\mu.
\]

Define
$$S_2=\{y\in B_R\setminus S_1: \dist(y,\chi_2\cap
B_R\setminus S_1)< \theta_0\}.$$
There exists $c_4>0$ such that
$\dist(\bar X_y(t), \chi_1)\geq 2c_4$ for any $y\in B_R\setminus S_1$.
Define
$$G_2=\{y\in B_R\setminus S_1: \dist(y,\chi_2\cap (B_R\setminus S_1))< c_4\}$$
There exists $c_5>0$ such that
$\dist(\bar X_y(t), \chi_1)\geq 2c_5$ for any $y\in B_R\setminus G_1$.
Define
$$N_2=\{y\in B_R\setminus S_1: \dist(y,\chi_1\cap B_R\setminus S_1)< c_5\}$$
Let $\hat T_2$ be such that
$\bar X_y(\hat T_2)\in B_R\setminus(S_1\cup S_2)$ given that $y\in B_R\setminus(S_1\cup N_2)$.
We can show, just as above, that
there exists a $T_{\Delta,2}^{\eps,\delta}$ such that $\lim_{\eps\to0} T_{\Delta,2}^{\eps,\delta}\exp\left(-\dfrac{\kappa}{\eps+\delta}\right)=0$ and
$$\PP\{\tau_{x,i}^{\eps,\delta}<T_{\Delta,2}^{\eps,\delta}\}>\frac12.$$
Define events
\begin{align*}
K_{1,2}^{\eps, \delta}&=\Big\{\widetilde X^{\eps, \delta}(T_{\Delta,2}^{\eps,\delta})\in N_2, \tau_2^{\eps,\delta}\geq T_{\Delta,2}^{\eps,\delta}, \widetilde X^{\eps,\delta}(0)\in N_2\Big\}\\
K_{2,2}^{\eps, \delta}&=\Big\{\widetilde X^{\eps, \delta}(T_{\Delta,2}^{\eps,\delta})\in N_2, \tau_2^{\eps,\delta}< T_{\Delta,2}^{\eps,\delta}, \widetilde X^{\eps,\delta}(0)\in N_2\Big\}\\
K_{3,2}^{\eps, \delta}&=\Big\{\widetilde X^{\eps, \delta}(T_{\Delta,2}^{\eps,\delta})\in N_2, \widetilde X^{\eps,\delta}(0)\in B_R\setminus (S_1\cup N_2)\Big\}\\
K_{4,2}^{\eps, \delta}&=\Big\{\widetilde X^{\eps, \delta}(T_{\Delta,2}^{\eps,\delta})\in N_2, \widetilde X^{\eps,\delta}(0)\notin B_R\setminus S_1\Big\}.
\end{align*}
Applying the same arguments as in the previous part,
we can show that $\limsup_{\eps\to 0} \mu^{\eps,\delta}(N_2)\leq 4\eta$.
Continuing this process, we can construct neighborhoods $N_1,\dots, N_{n_0-1}$ of $\chi_1\cap B_R$, $\dots$,$\chi_{n_0-1}\cap B_R$ such that
$$\limsup_{\eps\to0}  \mu^{\eps,\delta}(\cup_{j=1}^{n_0-1}N_j)\leq 2^{n_0}\eta.$$
\end{proof}
\main*

\begin{proof}
We have proved in Proposition \ref{p:nbhd} that for any $\eta>0$ we can find $R>0$ and neighborhoods $N_1,\dots, N_{n_0-1}$ of $\chi_1\cap B_R,\dots, \chi_{n_0-1}\cap B_R$ such that
$$\limsup_{\eps\to0}  \mu^{\eps,\delta}(\cup_{j=1}^{n_0}N_j)\leq 2^{n_0+1}\eta.$$
Using this fact together with Assumption \ref{asp1} and Lemma \ref{lm2.2},
by a straightforward modification of the proof of \cite[Theorem 1]{CH},
 we can establish
that for any $\vartheta>0$ there is neighborhood $N$ of the limit cycle $\Gamma$ such that
$$\liminf_{\eps\to  0} \mu^{\eps,\delta} (N)>1-\vartheta.$$
\end{proof}

\section{Proof of Theorem \ref{thm5.1}}\label{sec:5}

To proceed, we first need some auxiliary results.

\begin{lm}\label{lm5.1}
There exist numbers $ K_1, K_2>0$ such that for any $0<\eps,\delta<1$ and any $(i_0, z_0)\in\M\times\inte\R_+^{2}$,
we have
$$\dfrac1t\E\int_0^t|Z_{z_0, i_0}^{\eps,\delta}(s)|^2ds\leq
K_1(1+|z_0|),  t\geq1,$$
and
$$\limsup\limits_{t\to\infty}\E|Z_{z_0, i_0}^{\eps,\delta}(t)|^2\leq K_2.$$
\end{lm}

\begin{proof}
Let
$\theta<\min\{f_Mb(i), d(i): i\in\M\}$.
Define $$\hat  K_1=\sup\limits_{(x, y, i)\in\R^2_+\times\M}\{f_Mx(a(i)-b(i)x)-y(c(i)+d(i)y)+ \theta(x^2+y^2)\}<\infty.$$
Consider $\hat V(x, y, i)=f_M x+y.$
We can  check that $\mathcal{L}^{\eps,\delta}\hat V(x, y, i)\leq\hat  K_1-\theta(x^2+y^2),$
where $\mathcal{L}^{\eps,\delta}$ the generator associated with \eqref{ex1} (see \cite[p. 48]{MY} or \cite{YZ} for the formula of $\mathcal{L}^{\eps,\delta}$).
Similarly, we can verify that there is $\hat K_2>0$ such that for all $\eps<1,\delta<1$,
$\mathcal{L}^{\eps,\delta} (\hat V^2(x, y, i))\leq\hat  K_2-\hat V^2(x, y, i)$.
For each $k >0$,
define the stopping time $\sigma_k=\inf\{t: x(t)+y(t)>k\}.$
By the generalized It\^o  formula for $\hat V(x(t), y(t),\alpha^\eps(t))$
\begin{equation}\label{ex3}
\begin{aligned}
\E \hat V(Z_{z_0, i_0}^{\eps,\delta}(t\wedge\sigma_k), \alpha^\eps(t\wedge\sigma_k))
&= \hat V(z_0, i_0)+\E\int_0^{t\wedge\sigma_k}\mathcal{L}^{\eps,\delta}\hat V(Z_{z_0, i_0}^{\eps,\delta}(s), \alpha^\eps(s))ds\\
&\leq f_M x_0+y_0+\E\int_0^{t\wedge\sigma_k}\big[\hat K_1-\theta|Z_{z_0, i_0}^{\eps,\delta}(s)|^2\big]ds.
\end{aligned}
\end{equation}
Hence
$$\theta\E\int_0^{t\wedge\sigma_k}|Z_{z_0, i_0}^{\eps,\delta}(s)|^2ds\leq f_M x_0+y_0+\hat K_1t.$$
Letting $k\to\infty$ and dividing both sides by $\theta t$ we have
\begin{equation}\label{ex3b}
\dfrac1t\E\int_0^t|Z_{z_0, i_0}^{\eps,\delta}(s)|^2ds\leq \dfrac{f_M x_0+y_0}{\theta t}+\dfrac{\hat K_1}\theta.
\end{equation}
Applying the generalized It\^o  formula to $e^{t}\hat V^2(Z_{z_0, i_0}^{\eps,\delta}(t), \alpha^{\eps}(t))$,
\begin{equation}\label{ex3a}
\begin{aligned}
\E e^{t\wedge\sigma_k}&\hat V^2(Z_{z_0, i_0}^{\eps,\delta}(t\wedge\sigma_k), \alpha^{\eps}(t\wedge\sigma_k))\\
&= \hat V^2(z_0, i_0)+\E\int_0^{t\wedge\sigma_k}e^s\big[(\hat V^2(Z_{z_0, i_0}^{\eps,\delta}(s), \alpha^\eps(s))+\mathcal{L}^{\eps,\delta}\hat V^2(Z_{z_0, i_0}^{\eps,\delta}(s), \alpha^\eps(s))\big]ds\\
&\leq(f_M x_0+y_0)^2+\hat K_2\E\int_0^{t\wedge\sigma_k}e^sds\leq (f_M x_0+y_0)^2+\hat K_2 e^t.
\end{aligned}
\end{equation}
Taking the limit as $k\to\infty$, and then dividing both sides by $e^t$, we have
\begin{equation}\label{ex3c}
\E \big[f_MX_{z_0, i_0}^{\eps,\delta}(t)+Y_{z_0, i_0}^{\eps,\delta}(t)\big]^2\leq (f_M x_0+y_0)^2e^{-t}+\hat K_2.
\end{equation}
The assertions of the lemma follow directly from \eqref{ex3b} and \eqref{ex3c}.
\end{proof}

\begin{lm}\label{lm5.1a}
There is a number $K_3>0$ such that
$$\dfrac1t\E\int_0^t\Big[ \varphi^2(Z_{z, i}^{\eps,\delta}(s), \alpha_i^\eps(s))+\psi^2(Z_{z, i}^{\eps,\delta}(s), \alpha_i^\eps(s))\Big]ds\leq K_3(1+|z|)$$
for all $\eps,\delta\in (0, 1], z\in\inte\R_+^{2}, t\geq 1$.
\end{lm}

\begin{proof}
Since the function $h(\cdot, \cdot, i)$ is bounded, we can find $C>0$ such that
$$\varphi^2(z, i)+\psi^2(z, i)\leq C(1+|z|^2).$$
The claim follows by an application of Lemma \ref{lm5.1}.
\end{proof}

Recall that
the two equilibria of \eqref{ex2} on the boundary are both hyperbolic.
Note that the Jacobian of
$\Big(x\bar\phi(x,y), y\bar\psi(x,y)\Big)^\top$
at $\left(\frac{\bar a}{\bar b}, 0\right)$
has two eigenvalues: $-\bar c+\frac{\bar a}{\bar b}h_2\left(\frac{\bar a}{\bar b}, 0\right)>0$
and $-\frac{\bar b^2}{\bar a}<0$.
At $(0,0)$, the two eigenvalues are $\bar a>0$ and $-\bar c<0$, respectively.
If we consider the weighted average Lyapunov exponent,
we can see that the growth rate of $\dfrac{2\bar c}{\bar a} \frac{d\ln X(t)}{dt}+\frac{d\ln Y(t)}{dt}$
is positive both at $(0,0)$ and $\left(\frac{\bar a}{\bar b}, 0\right)$.
This suggests we should look at $\dfrac{2\bar c}{\bar a} \frac{d\ln X(t)}{dt}+\frac{d\ln Y(t)}{dt}$
in order to prove that the dynamics of \eqref{ex2} is pushed away from the boundary.
Then we can use approximation arguments to obtain the tightness of $(Z^{\eps,\delta})$ on $\inte\R_+^{2}$.
Define
$$\Upsilon(z,i):=\frac{2\bar c}{\bar a}\varphi(z,i)+\psi(z,i)$$
and
$$\bar\Upsilon(z):=\frac{2\bar c}{\bar a}\bar\varphi(z)+\bar\psi(z).$$
We have the following lemma.

\begin{lm}\label{lm5.2}
Let $\gamma_0=\dfrac12 \Big(\bar c\wedge \big(-\bar c+\frac{\bar a}{\bar b}h_1(\frac{\bar a}{\bar b}, 0)\big)\Big)>0.$ For any $H>\frac{\bar a}{\bar b}+1$, there are numbers $T, \beta>0$ such that for all $z\in\{(x,y)\in\R^2_+\,|\, x\wedge y\leq\beta,  x\vee y\leq H\}$
\begin{equation}\label{e0-lm5.2}
\bar X_z(T)\vee\bar Y_z(T)\leq H \text{ and } \dfrac1{T}\int_0^{T}\bar\Upsilon(\bar Z_z(t))dt\geq\gamma_0.
\end{equation}
\end{lm}

\begin{proof}
Since $\lim\limits_{t\to\infty}\bar Z_{(0, y)}(t)= (0,0), \ \forall  y\in\R_+$
and
\begin{equation}\label{e1-lm5.2}
\bar\Upsilon(0,0)=\frac{2\bar c}{\bar a}\bar\varphi(0,0)+\bar\psi(0,0)=\frac{2\bar c}{\bar a}\bar a- \bar c=\bar c\geq 2\gamma_0,
\end{equation}
there exists $T_1>0$ such that
\begin{equation}\label{e2-lm5.2}
\dfrac1t\int_0^t\bar\Upsilon(\bar Z_{(0, y)}(s))ds\geq\frac32\gamma_0\,\text{
for }\,t\geq T_1,\,y\in[0,H].
\end{equation}

By \eqref{e1-lm5.2} and the continuity of $\bar\Upsilon(\cdot)$, there exists $\beta_1\in (0, \frac{\bar a}{\bar b})$
such that
\begin{equation}\label{e3-lm5.2}
\bar\Upsilon(x,0)\geq \frac74\gamma_0,\text{ if } x\leq\beta_1.
\end{equation}
Since
$$\bar\Upsilon\left(\frac{\bar a}{\bar b},0\right)=\frac{2\bar c}{\bar a}\bar\varphi \left(\frac{\bar a}{\bar b},0\right)+\bar\psi \left(\frac{\bar a}{\bar b},0\right)=-\bar c+\frac{\bar a}{\bar b}h_1 \left(\frac{\bar a}{\bar b},0\right)\geq2\gamma_0$$
and
$$\lim\limits_{t\to\infty}\bar Z_{(x, 0)}(t)\to \left(\frac{\bar a}{\bar b},0\right), \ \forall  x>0,$$
there exists a $T_2>0$ such that
\begin{equation}\label{e5-lm5.2}
\dfrac1t\int_0^t\bar\Upsilon(\bar Z_{(x, 0)}(s))ds\geq\frac74\gamma_0\,\text{
for }\,t\geq T_2,\,x\in[\beta_1,H].
\end{equation}

Let $\bar M_H=\sup_{x\in[0,H]}\left\{|\bar\Upsilon(x,0)|\right\}
$,
$\bar t_{x}=\inf\{t\geq0: X_{x,0}\geq\beta_1\}$
and $T_3=\left(4\frac{\bar M_H}{\gamma_0}+7\right)T_2$.
It can be seen from the equation of $\bar X(t)$ that
$\bar X_{(x,0)}(t)\in[\beta_1,H]$
if $t\geq \bar t_x, x\in(0,\beta_1]$.
For $t\geq T_3$, we can use \eqref{e3-lm5.2} and \eqref{e5-lm5.2} to estimate $\frac1t\int_0^t\bar\Upsilon(\bar Z_{(x, 0)}(s))ds$ in
the following three cases.

{\bf Case 1}.  If $t-T_2\leq \bar t_x\leq t$
then
$$
\begin{aligned}
\int_0^t\bar\Upsilon(\bar Z_{(x, 0)}(s))ds
&=
\int_0^{\bar t_x}\bar\Upsilon(\bar Z_{(x, 0)}(s))ds
+\int_{\bar t_x}^t\bar\Upsilon(\bar Z_{(x, 0)}(s))ds\\
&\geq \frac74\gamma_0(t-T_2)-T_2\bar M_H\geq\frac32\gamma_0t,\,\,\bigg(\text{since }\, t\geq \Big(4\frac{\bar M_H}{\gamma_0}+7\Big)T_2\bigg).
\end{aligned}
$$

{\bf Case 2}.  If  $\bar t_x\leq t-T_2$,
then
$$
\begin{aligned}
\int_0^t\bar\Upsilon(\bar Z_{(x, 0)}(s))ds
&=
\int_0^{\bar t_x}\bar\Upsilon(\bar Z_{(x, 0)}(s))ds
+\int_{\bar t_x}^t\bar\Upsilon(\bar Z_{(x, 0)}(s))ds\\
&\geq \frac74\gamma_0(t-\bar t_x)+\frac74\gamma_0\bar t_x\geq\frac32\gamma_0t.
\end{aligned}
$$

{\bf Case 3}.  If  $\bar t_x\geq t$,
then
$$
\begin{aligned}
\int_0^t\bar\Upsilon(\bar Z_{(x, 0)}(s))ds
&=
\int_0^{\bar t_x}\bar\Upsilon(\bar Z_{(x, 0)}(s))ds
\geq \frac74\gamma_0\bar t_x\geq\frac32\gamma_0t.
\end{aligned}
$$
As a result,
\begin{equation}\label{e6-lm5.2}
\dfrac1t\int_0^t\bar\Upsilon(\bar Z_{(x, 0)}(s))ds\geq\frac32\gamma_0,\,\text{ if } t\geq T_3, x\in(0,H].
\end{equation}
Let $T=T_1\vee T_3$.
By the continuous dependence
of solutions on initial values,
there is $\beta>0$
such that
\begin{equation}\label{e7-lm5.2}
\bar X_z(T)\vee\bar Y_z(T)\leq H \text{ and } \dfrac1T\int_0^T\left|\bar\Upsilon(\bar Z_{z_1}(s))-\bar\Upsilon(\bar Z_{z_2}(s))\right|ds\leq \frac12\gamma_0
\end{equation}
given that $|z_1-z_2|\leq\beta, z_1,z_2\in[0,H]^2.$
Combining \eqref{e2-lm5.2}, \eqref{e6-lm5.2} and \eqref{e7-lm5.2}
we obtain the desired result.
\end{proof}

Generalizing the techniques in \cite{DY},
we divide the proof of the eventual tightness into two lemmas.

\begin{lm}\label{lm5.3}
For any $\Delta>0$, there exist $\eps_0, \delta_0, T>0$ and a compact set $\mathcal K\subset\inte\R_+^{2}$ such that
$$\liminf\limits_{k\to\infty}\dfrac1k\sum_{n=0}^{k-1}\PP\left\{Z^{\eps,\delta}_{z_0, i_0}(nT)\in \mathcal K\right\}\geq 1-\dfrac\Delta3\, \text{ for any }\, \eps<\eps_0, \delta<\delta_0, z\in\inte\R_+^{2}.$$
\end{lm}

\begin{proof}
For any $\Delta>0$, let $H=H(\Delta)>\frac{\bar a}{\bar b}+1$ be chosen later and define
$D=\{(x, y): 0<x, y\leq H\}$.
Let $T>0$ and $\beta>0$ such that \eqref{e0-lm5.2} is satisfied
and $D_1=\{(x, y): 0<x, y\leq H, x\wedge y<\beta\}\subset D$.
Define $V(x, y)=-\frac{2\bar c}{\bar a}\ln x-\ln y+ C$ where $C$ is a positive constant such that $V(z)\geq0\,\forall\, z\in D$. In view of the generalized It\^o  formula,
$$
\begin{aligned}
V(Z_{z, i}^{\eps,\delta}(t))-V(z)=&\int_0^t\left[-\Upsilon\big(Z_{z, i}^{\eps,\delta}(s), \alpha^\eps(s)\big)+\dfrac{\delta}2\left(\frac{2\bar c}{\bar a}\lambda^2(\alpha^\eps(s))+\rho^2(\alpha^\eps(s))\right)\right]ds\\
&-\frac{2\bar c}{\bar a}\int_0^t\sqrt{\delta}\lambda(\alpha^\eps(s))dW_1(s)-\int_0^t\sqrt{\delta}\rho(\alpha^\eps(s))dW_2(s).
\end{aligned}
$$
For $A\in\F$, using Holder's inequality and It\^o's isometry, we have	 	
\begin{equation}\label{ex11}
\begin{aligned}\E\Big(\1_{A}&\big|V(Z_{z, i}^{\eps,\delta}(T))-V(z)\big|\Big)\\
\leq&\left|\E\1_A\int_0^{T}\Upsilon\big(Z_{z, i}^{\eps,\delta}(t), \alpha^\eps(t)\big)dt\right|+\E\1_A\int_0^{T}\dfrac{\delta}2\left(\frac{2\bar c}{\bar a}\lambda^2(\alpha^\eps(t))+\rho^2(\alpha^\eps(t))\right)dt\\&+\frac{2\bar c}{\bar a}\E\1_A\left|\int_0^{T}\sqrt{\delta}\lambda(\alpha^\eps(t))dW_1(t)\right|+\E\1_A\left|\int_0^{T}\sqrt{\delta}\rho(\alpha^\eps(t))dW_2(t)\right|\\
\leq&T(\E\1_A)^{\frac12}\left(\E\int_0^{T}\left[\Upsilon\big(Z_{z, i}^{\eps,\delta}(t), \alpha^\eps(t)\big)+\dfrac\delta2\left(\frac{2\bar c}{\bar a}\lambda^2(\alpha^\eps(t))+\rho^2(\alpha^\eps(t))\right)\right]dt\right)^{\frac12}\\
&+\delta\sqrt{\PP(A)}\left(\E\int_0^T\left(\frac{2\bar c}{\bar a}\lambda^2(\alpha^\eps(t))+\rho^2(\alpha^\eps(t))\right)dt\right)^{\frac12}\\
\leq&  K_4T(1+|z|)\sqrt{\PP(A)},
\end{aligned}
\end{equation}
where the last inequality follows from \eqref{lm5.1a} and the boundedness of $\rho(i)$ and $\lambda(i)$.
If $A=\Omega$, we have
\begin{equation}\label{ex11a}
\dfrac1{T}\E\Big(\big|V(Z_{z, i}^{\eps,\delta}(T))-V(z)\big|\Big)\leq  K_4(1+|z|).
\end{equation}
Let $\hat H_{T}>H$ such that $\bar X_z(t)\vee \bar Y_z(t)\leq \hat H_{T}$ for all $z\in[0,H]^2,\, 0\leq t\leq {T}$ and
 $$\bar{d}_H=\sup\left\{\left|\dfrac{\partial \bar\Upsilon}{\partial x}(x, y)\right|, \left|\dfrac{\partial \bar\Upsilon}{\partial y}(x, y)\right|: (x,y)\in\R^2_+, x\vee y\leq \hat H_{T}\right\}.$$
Let $\varsigma>0$. Lemma \ref{lm2.2} implies that there are $\delta_0,\eps_0$ such that if $\eps<\eps_0,\delta<\delta_0$,
\begin{equation}\label{ex5}
\PP\left\{|\bar X_z(t)-X_{z, i}^{\eps,\delta}(t)|+|\bar Y_z(t)-Y_{z, i}^{\eps,\delta}(t)|<1\wedge\dfrac{\gamma_0}{2\bar{d}_H}, ~\text{for all}~  t\in[0,{T}]\right\}>1-\dfrac\varsigma6, \,z\in \bar D.
\end{equation}
On the other hand, if $|\bar X_z(t)-X_{z, i}^{\eps,\delta}(t)|+|\bar Y_z(t)-Y_{z, i}^{\eps,\delta}(t)|<1\wedge\dfrac{\gamma_0}{2\bar d_H}$, we have
\begin{equation}\label{ex6}
\begin{aligned}
\bigg|\dfrac{1}{T}\int_0^{T}&\Upsilon(Z_{z, i}^{\eps,\delta}(t), \alpha^\eps(t))dt
- \dfrac{1}{T}\int_0^{T}\bar\Upsilon(\bar Z_{z, i}(t))dt\bigg|\\
\leq&\dfrac1{T}\left|\int_0^{T}\Big(\bar\Upsilon(Z_{z, i}^{\eps,\delta}(t))-\bar\Upsilon(\bar Z_{z, i}(t))\Big)dt\right|\\
& +\dfrac1{T}\left|\int_0^{T}\Big(\Upsilon(Z_{z, i}^{\eps,\delta}(t), \alpha^\eps(t))-\bar\Upsilon(Z_{z, i}^{\eps,\delta}(t))\Big)dt\right|\\
\leq&\dfrac{\gamma_0}2+\dfrac{F_H}{T}\int_0^{T}\sum_{j\in\M}\big|\1_{\{\alpha^\eps(t)=j\}}-v_j\big|dt
\end{aligned}
\end{equation}
where
$F_H:=\sup\{|\Upsilon(z, i)| i\in\M, z\in[0, K_{T}+1]^2\}.$
In view of \cite[Lemma 2.1]{HYZ},
\begin{equation}\label{ex7}
\E\bigg|\dfrac1T\int_0^{T}\sum_{j\in\M}\big|\1_{\{\alpha^\eps(t)=j\}}-v_j\big|dt\bigg|^2=\E\bigg|\dfrac\eps{T}\int_0^{T/\eps}\sum_{j\in\M}\big|\1_{\{\alpha(t)=j\}}-v_j\big|dt\bigg|^2\leq \dfrac{\kappa}{T}\eps
\end{equation}	
for some constant $\kappa>0.$
On the one hand,
\begin{equation}\label{ex9}
\E\dfrac1{T}\left|\int_0^{T}\Big(-\frac{2\bar c}{\bar a}\lambda(\alpha^\eps(t))dW_1(t)-\rho(\alpha^\eps(t))dW_2(t)\Big)\right|^2\leq\frac{4\bar c^2}{\bar a^2}\lambda_M^2+\rho_M^2.
\end{equation}

Combining \eqref{e0-lm5.2}, \eqref{ex5}, \eqref{ex6}, \eqref{ex7}, and \eqref{ex9}, we can reselect $\eps_0$ and $\delta_0$ such that for $\eps<\eps_0,\delta<\delta_0$ we have
\begin{equation}\label{ex8}
\PP\left\{\dfrac{-1}{T}\int_0^{T}\Upsilon(\alpha^\eps(t), Z_{z, i}^{\eps,\delta}(t))dt\leq-0.5\gamma_0\right\}\geq 1-\dfrac\varsigma3,\,z\in D_1, i\in\M,
\end{equation}
\begin{equation}\label{ex8b}
\PP\Big\{X_{z, i}^{\eps,\delta}(T)\vee Y_{z, i}^{\eps,\delta}(T)\leq H \mbox{ (or equivalently } Z_{z, i}^{\eps,\delta}(T)\in D)\Big\}\geq1-\dfrac\varsigma3, z\in D_1
,\end{equation}
and
\begin{equation}\label{ex8a}
\PP\left\{\delta\vartheta+\dfrac{\sqrt{\delta}}{T}\left|\int_0^{T}\Big(\frac{2\bar c}{\bar a}\lambda(\alpha^\eps(t))dW_1(t)+\rho(\alpha^\eps(t))dW_2(t)\Big)dt\right|<0.25\gamma_0\right\}>1-\dfrac\varsigma3
\end{equation}
where
$\vartheta=\frac12\left(\frac{2\bar c}{\bar a}\lambda_M^2+\rho_M^2\right).$
Consequently, for any $(z, i)\in D_1\times\M$, there is a subset $\Omega_{z, i}^{\eps,\delta}\subset\Omega$ with $\PP(\Omega_{z, i}^{\eps,\delta})\geq1-\varsigma$
in which we have $Z_{z, i}^{\eps,\delta}(T)\in D$ and
\begin{equation}\label{ex10}
\begin{aligned}
\dfrac1{T}\big(V(Z_{z, i}^{\eps,\delta}({T}))-V(z)\big)\leq&\dfrac{-1}{T}\int_0^{T}\Upsilon(\alpha^\eps(t), Z_{z, i}^{\eps,\delta}(t))dt+\delta\vartheta\\&+\dfrac1{T}\Big|\int_0^{T}\sqrt{\delta}\Big(\frac{2\bar c}{\bar a}\lambda(\alpha^\eps(t))dW_1(t)+\rho(\alpha^\eps(t))dW_2(t)\Big)\Big|\\
\leq& -0.25\gamma_0
\end{aligned}
\end{equation}
On the other hand, we deduce from \eqref{ex11a} that for $z\in D$,
\begin{equation}\label{ex10a}
\PP\left\{\dfrac1{T}\big(V(Z_{z, i}^{\eps,\delta}({T}))-V(z)\big)\leq\Lambda\right\}\geq 1-\varsigma,
\end{equation}
where $\Lambda :=\frac{ K_4(1+2H)}\varsigma$.
Moreover, it also follows from \eqref{ex11a} that for $z\in D\setminus D_1$
$$\E V(Z_{z, i}^{\eps,\delta}({T})\leq \sup_{z\in D\setminus D_1}\big(V(z)+ K_4{T}|z|\big).$$
Define
\begin{equation}\label{ex10d}
L_1:=\sup_{z\in D\setminus D_1}V(z)+\Lambda T,~ L_2:=L_1+0.25\gamma_0,
\end{equation}
as well as $D_2:=\{(x, y)\in\inte\R_+^{2}: (x, y)\in D, V(x, y)>L_2\}$ and $U(z)=V(z)\vee L_1.$ It is clear that
\begin{equation}\label{ex12a}
U(z_2)-U(z_1)\leq|V(z_2)-V(z_1)| \text{ for any }z_1, z_2\in\R^{2\circ}_+.
\end{equation}
It follows from  \eqref{ex11}
that for any $\delta,\eps<1$, $A\in\F$, and $z\in D$, we have
\begin{equation}\label{ex12}
\dfrac1{T}\E\1_{A}\Big|V(Z_{z, i}^{\eps,\delta}(T))-V(z)\Big|\leq K_4(2H+1)\sqrt{\PP(A)}
.\end{equation}
Applying \eqref{ex12} and \eqref{ex12a} with $A=\Omega\setminus\Omega_{z, i}^{\eps,\delta}$,
we get
\begin{equation}\label{ex12b}
\dfrac1{T}\E\1_{\Omega\setminus\Omega_{z, i}^{\eps,\delta}}\Big[U(Z_{z, i}^{\eps,\delta}(T))-U(z)\Big]\leq K_4(2H+1)\sqrt{\varsigma},\text{ if } z\in D_1.
\end{equation}
In view of \eqref{ex10}, for
 $z\in D_2\text{ we have }V\big(Z_{z, i}^{\eps,\delta}(T)\big)<V(z)-0.25\gamma_0T.$
 By the definition of $D_2$,
 we also have $L_1\leq V(z)-0.25\gamma_0T.$
Thus, for any $z\in D_2$ and $\omega\in\Omega^{\eps,\delta}_{z, i}$
$$U\big(Z_{z, i}^{\eps,\delta}(T)\big)=L_1\vee V\big(Z_{z, i}^{\eps,\delta}(T)\big)\leq V(z)-0.25\gamma_0T=U(z)-0.25\gamma_0T,$$
which implies
\begin{equation}\label{ex10c}
\dfrac1{T}\Big[\E\1_{\Omega_{z, i}^{\eps,\delta}}U(Z_{z, i}^{\eps,\delta}({T}))-\E\1_{\Omega_{z, i}^{\eps,\delta}}U(z)\Big]\leq-0.25\gamma_0\PP(\Omega_{z, i}^{\eps,\delta})\leq -0.25\gamma_0(1-\varsigma).
\end{equation}
Combining \eqref{ex12b} with \eqref{ex10c}
\begin{equation}\label{ex13}
\dfrac1{T}\Big[\E U(Z_{z, i}^{\eps,\delta}({T}))-U(z)\Big]\leq-0.25\gamma_0(1-\varsigma)+ K_4(2H+1)\sqrt{\varsigma}, \ \forall  z\in D_2.
\end{equation}
For $z\in D_1\setminus D_2$,
 and $\omega\in\Omega_{z, i}^{\eps,\delta}$, we have from \eqref{ex10}  that $ V(Z_{z, i}^{\eps,\delta}({T}))\leq V(z)$. This shows that
$U(Z_{z, i}^{\eps,\delta}({T}))=L_1\vee V(Z_{z, i}^{\eps,\delta}({T}))\leq U(z)=V(z)\vee L_1.$
Hence, for $z\in D_1\setminus D_2$ and $\omega\in \Omega_{z, i}^{\eps,\delta}$ one has $$U(Z_{z, i}^{\eps,\delta}({T}))-U(z)\leq 0.$$
This and \eqref{ex12b} imply
\begin{equation}\label{ex14}
\dfrac1{T}\Big[\E U(Z_{z, i}^{\eps,\delta}({T}))-U(z)\Big]\leq K_4(2H+1)\sqrt{\varsigma}, \ \forall  z\in D_1\setminus D_2.
\end{equation}
If $z\in D\setminus D_1$, $U(z)=L_1$ and we have from \eqref{ex10a} and \eqref{ex10d} that
$$\PP\big\{U(Z_{z, i}^{\eps,\delta}({T}))=L_1\big\}= \PP\big\{V(Z_{z, i}^{\eps,\delta}({T}))\leq L_1\big\}\geq 1-\varsigma.$$
Thus
$$\PP\{U(Z_{z, i}^{\eps,\delta}({T}))=U(z)\}\geq 1-\varsigma.$$
Use \eqref{ex12} and \eqref{ex12a} again to arrive at
\begin{equation}\label{ex16}
\dfrac1{T}\Big[\E U(Z_{z, i}^{\eps,\delta}({T}))-U(z)\Big]\leq K_4(2H+1)\sqrt{\varsigma}, \ \forall \,z\in D\setminus D_1.
\end{equation}
On the other hand, equations \eqref{ex11a} and \eqref{ex12a} imply
\begin{equation}\label{ex17}
\dfrac1{T}\Big[\E U(Z_{z, i}^{\eps,\delta}({T}))-U(z)\Big]\leq  K_4(1+|z|), \   z\in\inte\R_+^{2}.
\end{equation}
Pick an arbitrary $(z_0, i_0)\in\inte\R_+^{2}\times\M$.
An application of the Markov property yields
$$
\begin{aligned}
\dfrac1{T}\Big[&\E U(Z_{z_0, i_0}^{\eps,\delta}((n+1){T}))-\E U(Z_{z_0, i_0}^{\eps,\delta}({nT}))\Big]\\
&=\sum_{i\in\M}\int_{\inte\R_+^{2}}\dfrac1{T}\Big[\E U(Z_{z, i}^{\eps,\delta}({T}))-U(z)\Big]\PP\Big\{Z_{z_0, i_0}^{\eps,\delta}(n{T})\in dz, \alpha^\eps(t)=i\Big\}.
\end{aligned}
$$
Combining \eqref{ex13}, \eqref{ex14},  \eqref{ex16}, and \eqref{ex17}, we get
$$
\begin{aligned}
\dfrac1{T}\Big[\E& U(Z_{z_0, i_0}^{\eps,\delta}((n+1){T}))-\E U(Z_{z_0, i_0}^{\eps,\delta}({nT}))\Big]\\
\leq&-\big[0.25\gamma_0(1-\varsigma)- K_4(2H+1)\sqrt{\varsigma}\big]\PP\big\{Z_{z_0, i_0}^{\eps,\delta}(n{T})\in D_2\big\}\\
&  + K_4(2H+1)\sqrt{\varsigma}\PP\big\{Z_{z_0, i_0}^{\eps,\delta}(n{T})\in D\setminus D_2\big\}\\
&+  K_4	\E\1_{\{Z_{z_0, i_0}^{\eps,\delta}(n{T})\notin D\}}\big(1+|Z_{z_0, i_0}^{\eps,\delta}(n{T})|\big)\\
\leq&-0.25\gamma_0(1-\varsigma)\PP\big\{Z_{z_0, i_0}^{\eps,\delta}(n{T})\in D_2\big\}+ K_4(2H+1)\sqrt{\varsigma}\\
&+  K_4\PP\big\{Z_{z_0, i_0}^{\eps,\delta}(nT)\notin D\big\}\E\big(1+|Z_{z_0, i_0}^{\eps,\delta}(nT)|\big).
\end{aligned}
$$

Note that
\bea \ad \liminf\limits_{k\to\infty}\dfrac1k\sum_{n=0}^{k-1}\dfrac1{T}\Big[\E U(Z_{z_0, i_0}^{\eps,\delta}((n+1){T}))-\E U(Z_{z_0, i_0}^{\eps,\delta}({nT}))\Big]\\
\aad \
=\liminf\limits_{k\to\infty}\dfrac1{kT}\E U(Z_{z_0, i_0}^{\eps,\delta}(k{T}))\geq0.
\eea
This forces
\begin{equation}\label{ex17a}
\begin{aligned}
0.25\gamma_0(1-\varsigma)&\limsup\limits_{k\to\infty}\dfrac1k\sum_{n=0}^{k-1}\PP\big\{Z_{z_0, i_0}^{\eps,\delta}(n{T})\in D_2\big\}\\
\leq & K_4(2H+1)\sqrt{\varsigma}+ K_3\limsup\limits_{k\to\infty}\dfrac1k\sum_{n=1}^k\PP\big\{Z_{z_0, i_0}^{\eps,\delta}(nT)\notin D\big\}\E\big(1+|Z_{z_0, i_0}^{\eps,\delta}(nT)|\big).
\end{aligned}
\end{equation}
In view of Lemma \ref{lm5.1}, we can choose $H=H(\Delta)$ independent of $(z_0, i_0)$ such that
\begin{equation}\label{ex18}
\limsup\limits_{t\to\infty}\PP\big\{Z_{z_0, i_0}^{\eps,\delta}(t)\notin D\big\}\leq \limsup\limits_{t\to\infty}\dfrac{\E(|Z_{z_0, i_0}^{\eps,\delta}(t)|)}
{H}\leq\dfrac{\Delta}{6}
,\end{equation}
and
\bea\ad  K_4\limsup\limits_{t\to\infty}\PP\big\{Z_{z_0, i_0}^{\eps,\delta}(t)\notin D\big\}\E\big(1+|Z_{z_0, i_0}^{\eps,\delta}(t)|\big)\\
\aad \ \leq  K_4\limsup\limits_{t\to\infty}\dfrac{\Big[\E\big(1+|Z_{z_0, i_0}^{\eps,\delta}(t)|\big)\Big]^2}{H}\leq\dfrac{0.1\gamma_0}{6}\Delta.
\eea
Hence, we have
\begin{equation}\label{ex19}
 K_4\limsup\limits_{k\to\infty}\dfrac1k\sum_{n=1}^k\PP\big\{Z_{z_0, i_0}^{\eps,\delta}(nT)\notin D\big\}\E\big(1+|Z_{z_0, i_0}^{\eps,\delta}(nT)|\big)\leq\dfrac{0.1\gamma_0}{6}\Delta.
\end{equation}
Choose $\varsigma=\varsigma(H)>0$ such that $0.25\gamma_0(1-\varsigma)\geq0.2\gamma_0$ and $ K_4(2H+1)\sqrt{\varsigma}\leq\dfrac{0.1\gamma_0}{6}\Delta$
and let $\eps_0=\eps_0(\varsigma, H),\delta_0(\varsigma, H)$ such that \eqref{ex8}, \eqref{ex8b}, and \eqref{ex8a} hold.
As a result, we get from \eqref{ex17a} and \eqref{ex19} that
\begin{equation}\label{ex20}
\limsup\limits_{k\to\infty}\dfrac1k\sum_{n=1}^k\PP\big\{Z_{z_0, i_0}^{\eps,\delta}(nT)\in D_2\big\}\leq \dfrac{\Delta}{6}.
\end{equation}
This together with \eqref{ex18} and \eqref{ex20} shows that for any $\eps<\eps_0, \delta<\delta_0$, we have
\begin{equation}\label{ex21}
\liminf\limits_{k\to\infty}\dfrac1k\sum_{n=1}^k\PP\big\{Z_{z_0, i_0}^{\eps,\delta}(nT)\in D\setminus D_2\big\}\geq1- \dfrac{\Delta}3.
\end{equation}
One can conclude the proof by noting that the set $D\setminus D_2$ is a compact subset of $\inte\R_+^{2}$.
\end{proof}

\begin{lm}\label{lm5.4}
There are $ L>1$, $\eps_1=\eps(\Delta)>0$, and $\delta_1=\delta_1(\Delta)>0$ such that as long as $0<\eps<\eps_1, 0<\delta<\delta_1$, we have
$$\liminf\limits_{T\to\infty}\dfrac1T\int_0^T\PP\{Z^{\eps,\delta}_{z_0, i_0}(t)\in [ L^{-1}, L]^2\}\geq 1-\Delta,   (z_0, i_0)\in\inte\R_+^{2}\times\M.$$
\end{lm}

\begin{proof}
Let $D$ and $T$ as in Lemma \ref{lm5.3}.
Since $D\setminus D_2$ is a compact set in $\inte\R_+^{2}$, by a modification of the proof of \cite[Theorem 2.1]{JJ}, we can  show that there is a positive constant $ L>1$ such that
$\PP\{Z_{z, i}(t)\in [ L^{-1}, L]^2\}>1-\dfrac\Delta3,   z\in D\setminus D_2, i\in M, 0\leq t\leq T$.
Hence, it follows from the Markov property of the solution that
\bea \ad\PP\{Z_{z_0, i_0}^{\eps, \delta}(t)\in[ L^{-1}, L]^2\}\\
\aad\quad \geq \Big(1-\dfrac{\Delta}3\Big)\PP\big\{Z_{z_0, i_0}^{\eps,\delta}(jT)\in D\setminus D_2\big\}, t\in[jT, jT+T].\eea
Consequently,
\begin{align*}
\liminf\limits_{k\to\infty}\dfrac1{kT}&\int_{0}^{kT}\PP\big\{Z_{z_0, i_0}^{\eps,\delta}(t)\in [ L^{-1}, L]^2\big\}\,dt\\&\geq \Big(1-\dfrac{\Delta}3\Big)\liminf\limits_{k\to\infty}\dfrac1{k}\sum_{j=0}^{k-1}\PP\big\{Z_{z_0, i_0}^{\eps,\delta}(jT)\in D\setminus D_2)\big\}\geq 1-\Delta.
\end{align*}
It is readily seen from this estimate that
$$\liminf\limits_{T\to\infty}\dfrac1{T}\int_{0}^{T}\PP\big\{Z_{z_0, i_0}^{\eps,\delta}(t)\in [ L^{-1}, L]^2\big\}dt\geq 1-\Delta.$$
\end{proof}

\mainnex*
\begin{proof}
The conclusion of Lemma \ref{lm5.4}
 is sufficient for the existence of a unique invariant probability measure
 $\mu^{\eps,\delta}$ in $\inte\R_+^{2}\times\M$ of $(Z^{\eps,\delta}(t), \alpha^\eps(t))$ (see \cite{LB} or \cite{MT}). Moreover, the empirical measures
 $$\dfrac1t\int_0^t\PP\big\{Z^{\eps,\delta}_{z_0,i_0}(s)\in \cdot\big\}ds, t>0$$
 converge weakly to the invariant probability measure $\mu^{\eps,\delta}$ as $t\to\infty$.
Applying Fatou's lemma to the above estimate yields
$$\mu^{\eps,\delta}([ L^{-1}, L]^2)\geq \Delta, \ \forall \,\eps<\eps_0, \delta<\delta_0.$$
This tightness implies Theorem \ref{thm5.1}.
\end{proof}

\subsection{An Example}
%{Numerical Examples}
\label{sec:6}
In this section we provide a specific %numerical
example
%for
under the setting
of Section \ref{sec:5}.
We consider the following stochastic predator-prey model with Holling functional response in a switching regime
\begin{equation}\label{nex1}
\left\{\begin{array}{ll} \disp d\xx (t)&\!\!\! \disp =\bigg[r(\alpha^\eps(t))\xx(t)\left(1-\dfrac{\xx(t)}{K(\alpha^\eps(t)}\right)-\dfrac{m(\alpha^\eps(t))\xx(t)\yy(t)}{a(\alpha^\eps(t))
+b(\alpha^\eps(t))\xx(t)}\bigg]dt\\
& \quad \ +\sqrt{\delta}\lambda(\alpha^\eps(t))\xx(t)dW_1(t)\\
\disp d\yy(t)&\!\!\! \disp =\yy(t)\bigg[-d(\alpha^\eps(t))+\dfrac{e(\alpha^\eps(t))
m(\alpha^\eps(t))\xx(t)}{a(\alpha^\eps(t))+b(\alpha^\eps(t))\xx(t)}-f(\alpha^\eps(t))\yy(t)\bigg]dt\\
&\quad \ +\sqrt{\delta}\rho(\alpha^\eps(t))\xx(t)dW_2(t),\end{array}\right.
\end{equation}
where $W_1$ and $W_2$ are two independent Brownian motions, $\alpha^\eps(t)$ is a Markov chain, that is
 independent of the Brownian motions, with state space $\M=\{1, 2\}$ and generator $Q/\eps$  where
\begin{equation}\label{eq-q}
Q =
\left( \begin{array}{rr}  -1 &  1 \\
 1  &  -1 \\ \end{array} \right),
 \end{equation}
and $r(1)=0.9, r(2)=1.1, K(1)=4.737, K(2)=5.238, m(1)=1.2, m(2)=0.8, a(1)=a(2)=1, b(1)=b(2)=1, d(1)=0.85, d(2)=1.15, e(1)=1, e(2)=2.5, f(1)=0.03, f(2)=0.01, \lambda(1)=1, \lambda(2)=2, \rho(1)=3, \rho(2)=1$.
As $\eps$ and $\delta$ tend to 0, solutions of equation \eqref{nex1} converge  to the corresponding solutions of
\begin{equation}\label{nex2}
\left\{\begin{array}{lll}\disp {d \over dt} {x}(t)=x(t)\left(1-\dfrac{x(t)}{5}\right)-\dfrac{x(t)y(t)}{1+x(t)},
\\
\disp {d \over dt} {y}(t)=y(t)\left(-1+\dfrac{1.6x(t)}{1+x(t)}-0.02y(t)\right)
\end{array}\right.
\end{equation}
on any finite time interval $[0,T]$. The system \eqref{nex2}  has the unique equilibrium $(x^*, y^*)=(1.836, 1.795)$.
Modifying \cite[Theorem 2.6]{SR} it can be seen that the solution of equation
\eqref{nex2} has a unique limit cycle $\Gamma$ that attracts all positive solutions except for
 $(x^*, y^*)$.
Moreover, it is easy to check that the drift
 \begin{equation}\label{nex4}
\left(\begin{array}{l}r(i)x(t)\left(1-\dfrac{x(t)}{K(i)}\right)-\dfrac{m(i)x(t)y(t)}{a(i)+b(i)x(t)}\\
y(t)\left(-d(i)-\dfrac{e(i)m(i)x(t)}{a(i)+b(i)x(t)}\right)-f(i)y(t)\end{array}\right)
\end{equation}
does not vanish at $(1.836, 1.795)$. The assumptions of Theorem \ref{thm5.1} hold in this example. As a result, the family $(\mu^{\eps,\delta})_{\eps>0}$ converges weakly as $\eps\to  0$ to the stationary distribution of \eqref{nex2} that is concentrated on the limit cycle $\Gamma$.
We illustrate this convergence in Figures \ref{f1}, \ref{f2} and \ref{f3} below by graphing sample paths of \eqref{nex1} for different values of $(\eps,\delta)$.

\begin{figure}[h]
\centering
\includegraphics[totalheight=2.2in,width=2.1in]{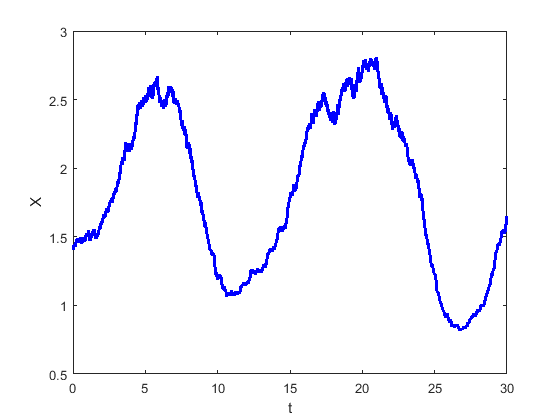}
\includegraphics[totalheight=2.2in,width=2.1in]{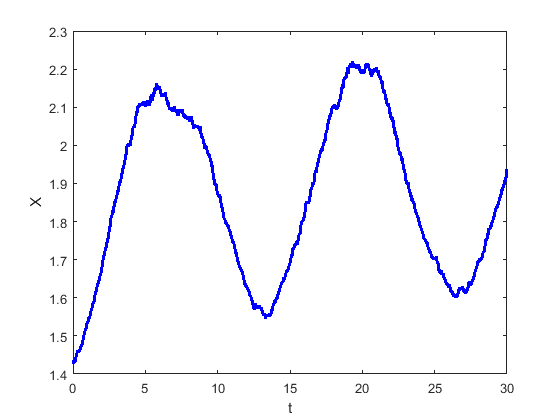}
\includegraphics[totalheight=2.2in,width=2.1in]{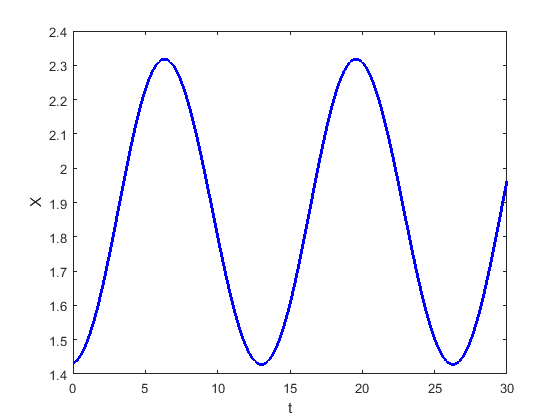}
\caption{From left to right: Graphs of the $x^{\eps,\delta}(t)$ component of \eqref{nex1} with $(\eps,\delta)=(0.001, 0.001)$, $(\eps,\delta)=(0.00005, 0.00005)$ and $x(t)$ of the averaged system \eqref{nex2} respectively.}
\label{f1}
\end{figure}

\begin{figure}[h]
\centering
\includegraphics[totalheight=2.2in,width=2.1in]{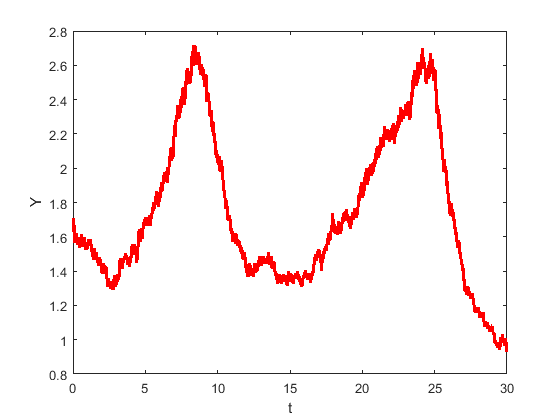}
\includegraphics[totalheight=2.2in,width=2.1in]{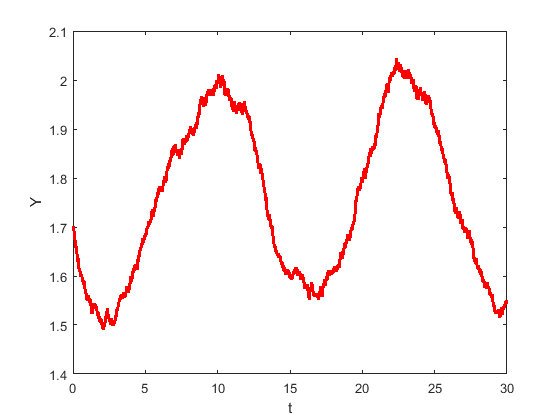}
\includegraphics[totalheight=2.2in,width=2.1in]{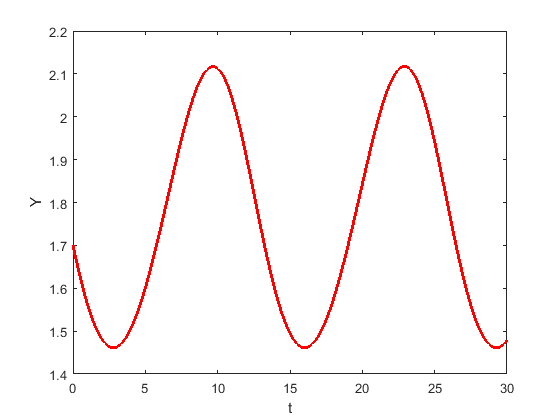}
\caption{From left to right: Graphs of the $y^{\eps,\delta}(t)$ component of \eqref{nex1} with $(\eps,\delta)=(0.001, 0.001)$, $(\eps,\delta)=(0.00005, 0.00005)$ and $y(t)$ of the averaged system \eqref{nex2} respectively.}
\label{f2}
\end{figure}
\begin{figure}[h]
\centering
\includegraphics[totalheight=2.2in,width=2.1in]{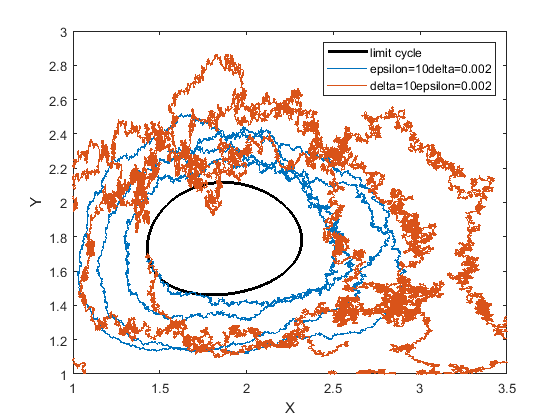}
\includegraphics[totalheight=2.2in,width=2.1in]{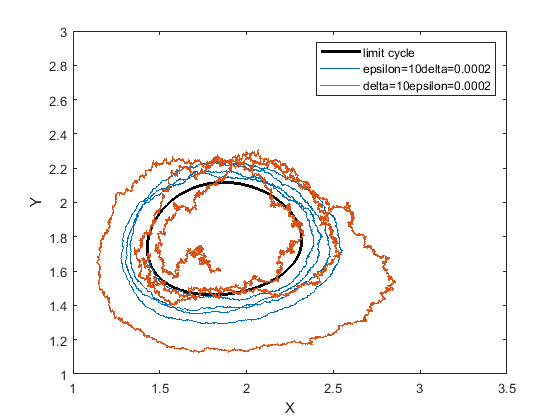}
\includegraphics[totalheight=2.2in,width=2.1in]{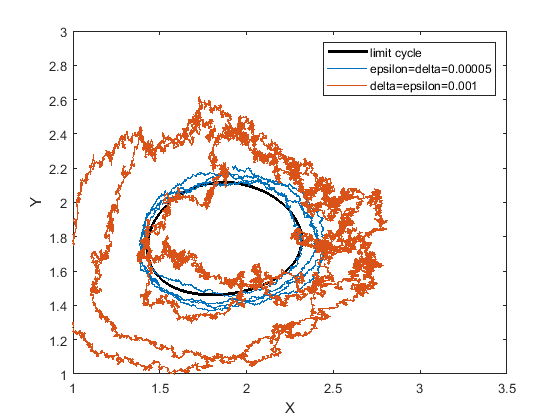}
\caption{Phase portraits of \eqref{nex1} for different values of $\eps$ and $\delta$.}
\label{f3}
\end{figure}
\begin{figure}[h]
	\centering
	\includegraphics[totalheight=3.2in,width=3.1in]{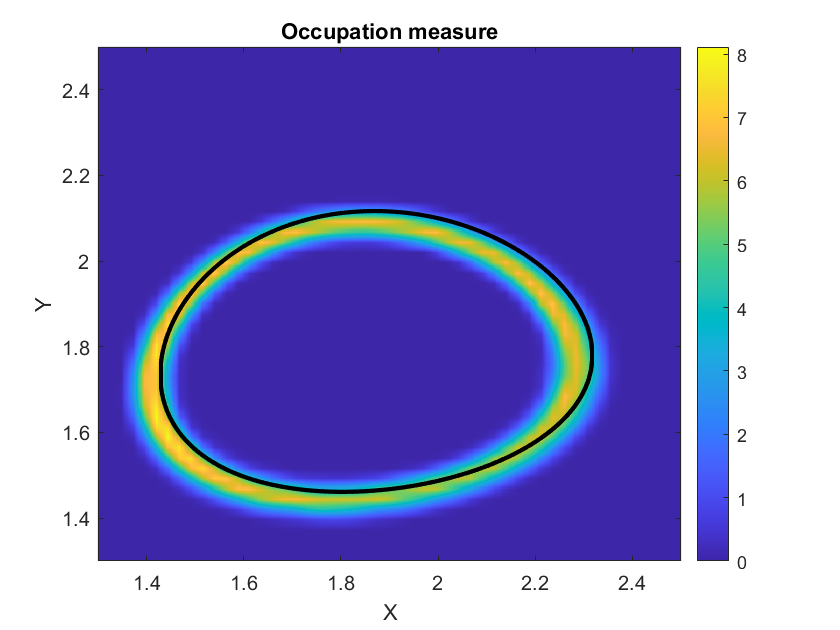}
	\includegraphics[totalheight=3.2in,width=3.1in]{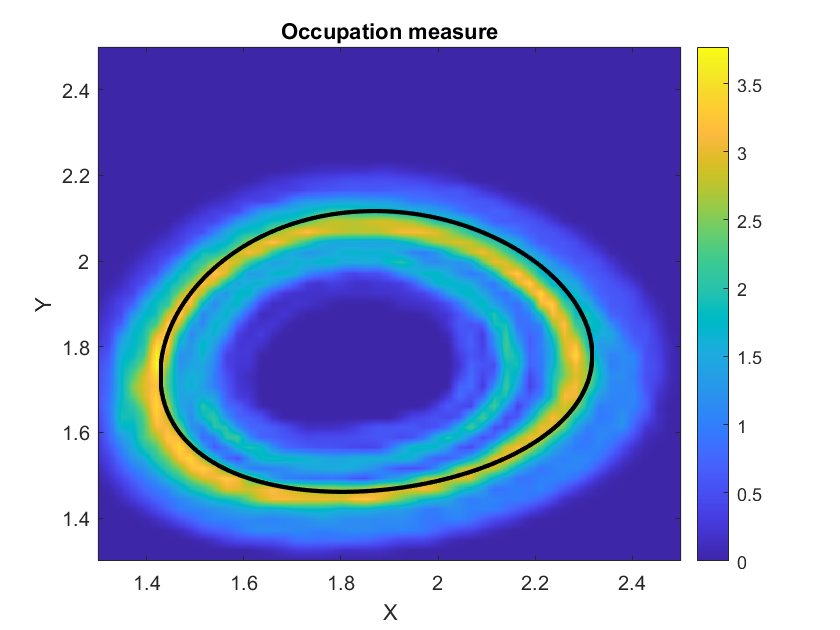}
	\caption{The two figures illustrate an approximate density of the occupation measure on the time interval $[0,1000]$ with $\eps=\delta=0.001$ on the right and $\eps=\delta=0.0001$ on the left. The step size is $h=0.0001$. We divide the domain in $50\times 50$ cells, approximate the density by the frequency of the process staying in each cell, and then interpolate. The simulations support the theoretical results that the occupation measures converge to the invariant probability measure.}
	\label{f3}
\end{figure}
\subsection{Another Example}
Consider the following system
\begin{equation}\label{ex2-e1}
\begin{cases}
d\xx(t)=& \left[-\yy(t)+\xx\left(1-a(\ale(t))\left((\xx(t))^2+(\yy(t))^2\right)\right)(1+(\zz(t))^2)\right]dt\\
&+\sqrt{\delta} dW_1(t)\\
d\yy(t)=&\left[ -\xx(t)+\yy\left(1-a(\ale(t))*\left((\xx(t))^2+(\yy(t))^2\right)\right)(1+(\zz(t))^2)\right]\\
&+\sqrt{\delta} dW_2(t)\\
d\zz(t)=&\left[\zz(t)\left(1-(\zz(t))^2-b(\ale(t)\left((\xx(t))^2+(\yy(t))^2\right)\right)\right]dt+\sqrt{\delta} dW_3(t)
\end{cases}\end{equation}
with $a(1)=0.8, a(2)=1.2$, $b(1)=3.5$ and $b(2)=4.5$ and $Q$ defined as in \eqref{eq-q}.
The limit system is
\begin{equation}
\begin{cases}
\frac{dx}{dt}=&-y+x(1-x^2-y^2)(1+z^2)\\
\frac{dy}{dt}=&x+y(1-x^2-y^2)(1+z^2)\\
\frac{dz}{dt}=&z(1-z^2-4x^2-4y^2)\\
\end{cases}
\end{equation}
which has a unique limit cycle $\{z=0, x^2+y^2=1\}$ and three hyperbolic equilibria: $(0,0,0)$ whose eigenvalues have both negative and positive parts and $(0, 0, \pm 1)$ which are sources.
\begin{figure}[h]
	\centering
	\includegraphics[totalheight=3.2in,width=3.1in]{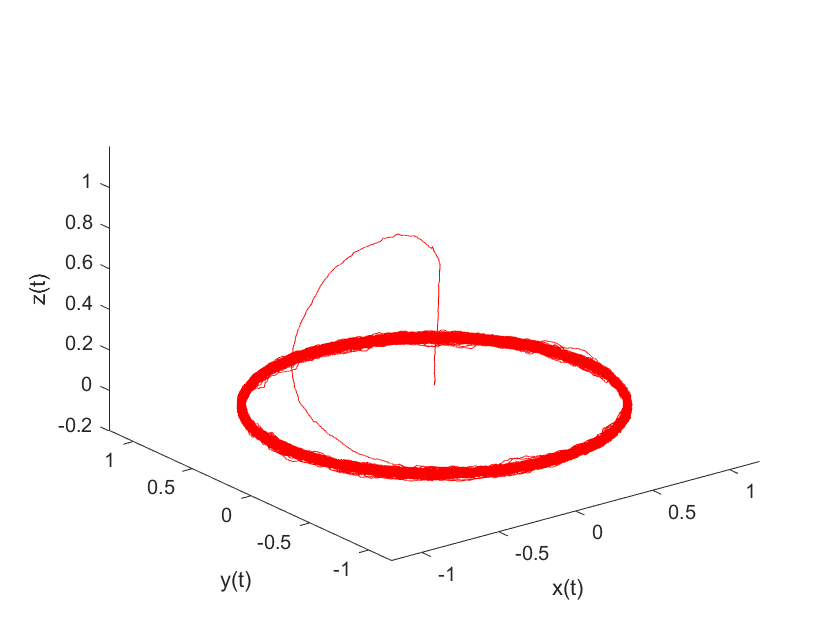}
	\includegraphics[totalheight=3.2in,width=3.3in]{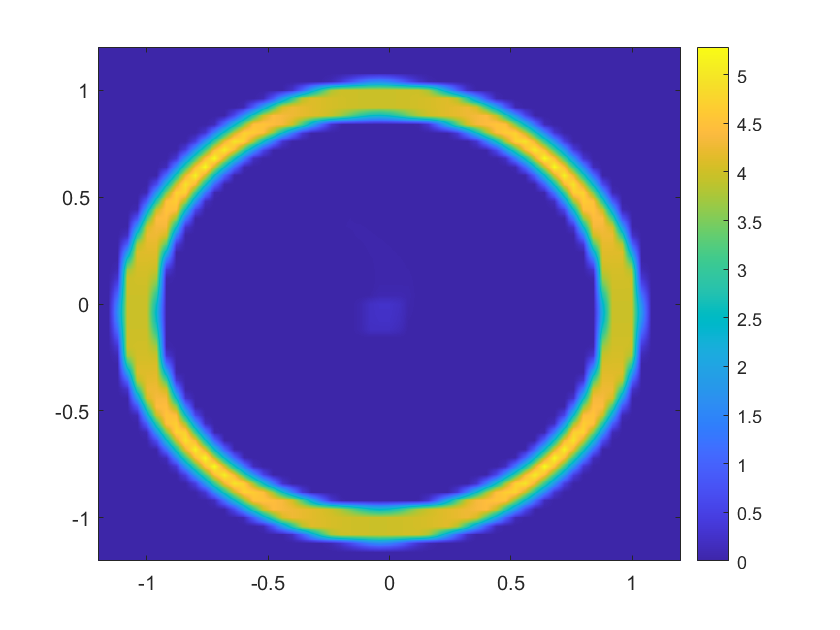}
	\caption{The left figure is a sample path starting at the equilibrium $(0,0,0)$ of the limit system. We see that this path has a tendency of approaching $(0,0,1)$ before it is finally attracted to the proximity of the limit cycle. The right figure is an approximated density of the occupation measure of two components $x,y$ on the time interval $[0,1000]$.}
	\label{f3}
\end{figure}
\iffalse
{\bf Acknowledgments.}
We thank a referee who read the original version of the manuscript and suggested us
 to examine a more challenging problem, namely, to study hyperbolic equilibria. This has much improved the paper.

The research of
Nguyen H. Du was supported in part  by
NAFOSTED n$_0$ 101.02 - 2011.21;
the research of
Alexandru Hening was supported in part  by the National Science Foundation under grant DMS-1853463,
the research of Dang Nguyen
was supported in part
by the National Science Foundation under grant DMS-1853467;
 the research of G. Yin was supported in part by the National Science Foundation under grant DMS-1710827.
\fi

\clearpage

\bibliographystyle{amsalpha}
\bibliography{cycle}

\newcommand{\etalchar}[1]{$^{#1}$}
\providecommand{\bysame}{\leavevmode\hbox to3em{\hrulefill}\thinspace}
\providecommand{\MR}{\relax\ifhmode\unskip\space\fi MR }
\providecommand{\MRhref}[2]{%
  \href{http://www.ams.org/mathscinet-getitem?mr=#1}{#2}
}
\providecommand{\href}[2]{#2}
\begin{thebibliography}{BDG{\etalchar{+}}18}

\bibitem[Bak08]{B08}
Yuri Bakhtin, \emph{Exit asymptotics for small diffusion about an unstable
  equilibrium}, Stochastic Processes and their Applications \textbf{118}
  (2008), no.~5, 839--851.

\bibitem[Bak11]{B11}
\bysame, \emph{Noisy heteroclinic networks}, Probability theory and related
  fields \textbf{150} (2011), no.~1-2, 1--42.

\bibitem[BDG{\etalchar{+}}18]{budhiraja2018large}
Amarjit Budhiraja, Paul Dupuis, Arnab Ganguly, et~al., \emph{Large deviations
  for small noise diffusions in a fast markovian environment}, Electronic
  Journal of Probability \textbf{23} (2018).

\bibitem[Bel06]{LB}
L.~R. Bellet, \emph{Ergodic properties of {M}arkov processes}, In {\it Open
  Quantum Systems II}, Springer Berlin Heidelberg, 1-39, 2006.

\bibitem[Ben18]{B18}
Michel Benaim, \emph{Stochastic persistence}, arXiv preprint arXiv:1806.08450
  (2018).

\bibitem[BL16]{BL16}
M.~Bena{\"\i}m and C.~Lobry, \emph{Lotka--volterra with randomly fluctuating
  environments or “how switching between beneficial environments can make
  survival harder”}, The Annals of Applied Probability \textbf{26} (2016),
  no.~6, 3754--3785.

\bibitem[CC01]{CC01}
R.~S. Cantrell and C.~Cosner, \emph{On the dynamics of predator--prey models
  with the beddington--deangelis functional response}, Journal of Mathematical
  Analysis and Applications \textbf{257} (2001), no.~1, 206--222.

\bibitem[Con78]{C78}
Charles~C Conley, \emph{Isolated invariant sets and the morse index}, no.~38,
  American Mathematical Soc., 1978.

\bibitem[Cos16]{C16}
M.~Costa, \emph{A piecewise deterministic model for a prey-predator community},
  The Annals of Applied Probability \textbf{26} (2016), no.~6, 3491--3530.

\bibitem[Dav84]{D84}
Mark H.~A. Davis, \emph{Piecewise-deterministic {M}arkov processes: A general
  class of non-diffusion stochastic models}, Journal of the Royal Statistical
  Society. Series B (Methodological) (1984), 353--388.

\bibitem[Day82]{D82}
Martin Day, \emph{Exponential leveling for stochastically perturbed dynamical
  systems}, SIAM Journal on Mathematical Analysis \textbf{13} (1982), no.~4,
  532--540.

\bibitem[DDT11]{DDT}
N.~H. Dang, N.~H. Du, and T.~V. Ton, \emph{Asymptotic behavior of predator-prey
  systems perturbed by white noise}, Acta Appl. Math. \textbf{115} (2011),
  no.~3, 351--370.

\bibitem[DNY16]{DDY}
N.~H. Du, D.~H. Nguyen, and G.~Yin, \emph{Conditions for permanence and
  ergodicity of certain stochastic predator-prey models}, J. Appl. Probab
  \textbf{53} (2016), no.~1, 187--202.

\bibitem[DS12]{DS12}
Paul Dupuis and Konstantinos Spiliopoulos, \emph{Large deviations for
  multiscale diffusion via weak convergence methods}, Stochastic Processes and
  their Applications \textbf{122} (2012), no.~4, 1947--1987.

\bibitem[DSW12]{DSW12}
Paul Dupuis, Konstantinos Spiliopoulos, and Hui Wang, \emph{Importance sampling
  for multiscale diffusions}, Multiscale Modeling \& Simulation \textbf{10}
  (2012), no.~1, 1--27.

\bibitem[Fle74]{WF}
W.~Fleming, \emph{Stochastically perturbed dynamical systems}, Rocky Mountain
  J. Maths \textbf{4} (1974), 407--433.

\bibitem[FW98]{FW98}
Mark~Iosifovich Freidlin and Alexander~D Wentzell, \emph{Random perturbations},
  Random perturbations of dynamical systems, Springer, 1998, pp.~15--43.

\bibitem[GH79]{GH79}
T.~C. Gard and T.~G. Hallam, \emph{Persistence in food webs. {I}.
  {L}otka-{V}olterra food chains}, Bull. Math. Biol. \textbf{41} (1979), no.~6,
  877--891. \MR{640001}

\bibitem[HN18]{HN16}
A.~Hening and D.~H. Nguyen, \emph{Coexistence and extinction for stochastic
  {K}olmogorov systems}, Ann. Appl. Probab. \textbf{28} (2018), no.~3,
  1893--1942.

\bibitem[HN20]{HN18}
A.~Hening and D.~H. Nguyen, \emph{The competitive exclusion principle in
  stochastic environments}, Journal of Mathematical Biology \textbf{80} (2020),
  1323--1351.

\bibitem[Hol74]{CH1}
C.~J. Holland, \emph{Ergodic expansions in small noise problems}, J.
  Differential Equations \textbf{16} (1974), 281--288.

\bibitem[Hol78]{CH}
\bysame, \emph{Stochastically perturbed limit cycles}, J. Appl. Probab
  \textbf{15} (1978), no.~2, 311--320.

\bibitem[HY14]{HY14}
Q.~He and G.~Yin, \emph{Large deviations for multi-scale {M}arkovian switching
  systems with a small diffusion}, Asymptotic Anal \textbf{87} (2014),
  123--145.

\bibitem[HYZ11]{HYZ}
Q.~He, G.~Yin, and Q.~Zhang, \emph{Large deviations for two-time-scale systems
  driven by nonhomogeneous {M}arkov chains and associated optimal control
  problems}, SIAM J. Control Optim \textbf{49} (2011), no.~4, 1737--1765.

\bibitem[JJ11]{JJ}
C.~Ji and D.~Jiang, \emph{{Dynamics of a stochastic density dependent
  predator-prey system with {B}eddington-{D}e{A}ngelis functional response}},
  J. Math. Anal. Appl \textbf{381} (2011), 441--453.

\bibitem[JJNS11]{CDN}
C.~Ji, D.~Jiang, and A~N.~Shi, \emph{Note on a predator-prey model with
  modified {L}eslie-{G}ower and {H}olling type ii schemes with stochastic
  perturbation}, J. Math. Anal. Appl \textbf{377} (2011), 435--440.

\bibitem[Kif81]{K81}
Yuri Kifer, \emph{The exit problem for small random perturbations of dynamical
  systems with a hyperbolic fixed point}, Israel Journal of Mathematics
  \textbf{40} (1981), no.~1, 74--96.

\bibitem[Kif12]{kifer12}
\bysame, \emph{Random perturbations of dynamical systems}, vol.~16, Springer
  Science \& Business Media, 2012.

\bibitem[Lot25]{L25}
A.~J. Lotka, \emph{Elements of physical biology}.

\bibitem[Mao07]{XM}
X.~Mao, \emph{Stochastic differential equations and applications}, Elsevier,
  2007.

\bibitem[MT93]{MT}
S.~P. Meyn and R.~L. Tweedie, \emph{Stability of {M}arkovian processes iii:
  {F}oster--{L}yapunov criteria for continuous-time processes}, Adv. Appl. Prob
  \textbf{25} (1993), 518--548.

\bibitem[MY06]{MY}
X.~Mao and C.~Yuan, \emph{Stochastic differential equations with {M}arkovian
  switching}, Imperial College Press, London, 2006.

\bibitem[NY17]{DY}
D.~H. Nguyen and G.~Yin, \emph{Coexistence and exclusion of stochastic
  competitive {L}otka--{V}olterra models}, J. Differential Eqs \textbf{262}
  (2017), no.~3, 1192--1225.

\bibitem[Per13]{perko13}
Lawrence Perko, \emph{Differential equations and dynamical systems}, vol.~7,
  Springer Science \& Business Media, 2013.

\bibitem[SR93]{SR}
A.~Sikder and A.~B. Roy, \emph{Limit cycles in a prey-predator system}, Appl.
  Math. Lett. \textbf{6} (1993), no.~3, 91--95.

\bibitem[VF70]{FW70}
AD~Ventcel and Mark~Iosifovich Freidlin, \emph{On small random perturbations of
  dynamical systems}, Russian Math. Surveys \textbf{25} (1970), no.~1, 1--55.

\bibitem[Vol28]{V28}
V.~Volterra, \emph{Variations and fluctuations of the number of individuals in
  animal species living together}, J. Cons. Int. Explor. Mer \textbf{3} (1928),
  no.~1, 3--51.

\bibitem[YZ10]{YZ}
G.~Yin and C.~Zhu, \emph{Hybrid switching diffusions: Properties and
  applications}, Springer, 2010.

\end{thebibliography}

\appendix
\section{Proofs of Lemmas from Section \ref{sec:3}}\label{a:1}
\begin{lm}\label{lm2.1}
For any $R, T, \gamma>0$, there exists a number $k_1=k_1(R, T, \gamma)>0$ such that for all sufficiently small $\delta$,
$$\PP\{|X^{\eps,\delta}_{x, i}(t)-\xi^{\eps}_{x, i}(t)|\geq\gamma,  ~\text{for some}~t\in [0,T]\}<\exp\left(-\dfrac{k_1}{\delta}\right), x\in B_R,$$
where $X^{\eps,\delta}_{x, i}(t)$ and $\xi^{\eps}_{x, i}(t)$ are the solutions to
the systems \eqref{eq2.1} and \eqref{eq2.3} that have initial value $(x, i)$.
\end{lm}

\begin{proof}
By (i) and (ii) of Assumption \ref{asp1}, we can deduce the existence and boundedness of a unique solution to equation \eqref{eq2.3}
using
the
Lyapunov function method. Moreover, we can find $R_T>R>0$ such that almost surely
\begin{equation}\label{e:bound}
|\xi^{\eps}_{x, i}(t)|<R_T-\gamma, ~\text{for all}~t\in[0, T], x\in B_R.
\end{equation}
Let $h(\cdot)$ be a twice differentiable function with compact support such that $h(x)=1$ if $|x|\leq R_T$ and $h(x)=0$
if $|x|\geq R_T+1$.
Put $f_h(x, i)=h(x)f(x, i)$, $\sigma_h(x, i)=h(x)\sigma(x, i)$ and
let $Y^{\eps,\delta}_{x, i}(t)$ be the solution starting at $(x, i)$ of
\begin{equation}
d Y(t)= f_h(Y(t),\alpha^\eps(t))dt+\sqrt{\delta}\sigma_h(Y(t), \alpha^\eps(t)dW(t)
\end{equation}
Note that $Y^{\eps,\delta}_{x, i}(t)=X^{\eps,\delta}_{x, i}(t)$ up to the time $\zeta=\inf\{t>0: |X^{\eps,\delta}_{x, i}(t)|>R_T\}$.
Because of \eqref{e:bound},
the solution  $\xi^{\eps}_{x, i}(t)$ to \eqref{eq2.3} coincides with the solution to
$$d Z(t)= f_h(Z(t),\alpha^\eps(t))dt$$ with
 starting point $x\in B_R$ and  $t\in [0,T]$.
We have from the generalized It\^o's formula that
for all $x\in B_R$ and $t\in [0,T]$,
\begin{equation}
\begin{aligned}
|Y&^{\eps,\delta}_{x, i}(t)-\xi^{\eps}_{x, i}(t)|^2\\
\leq&2\int_0^{t}|Y^{\eps,\delta}_{x, i}(s)-\xi^{\eps}_{x, i}(s)||f_h(Y^{\eps,\delta}_{x, i}(s),\alpha^\eps(s))-f_h(\xi^{\eps}_{x, i}(s),\alpha^\eps(s))|ds\\
&+\int_0^{t}\delta\trace\big((\sigma_h\sigma_h')(Y^{\eps,\delta}_{x, i}(s),\alpha^\eps(s))\big)ds\\
&+2\sqrt{\delta}\left|\int_0^{t}\big(Y^{\eps,\delta}_{x, i}(s)-\xi^{\eps}_{x, i}(s)\big)'\sigma_h(Y^{\eps,\delta}_{x, i}(s),\alpha^\eps(s)\big)dW(s)\right|.
\end{aligned}
\end{equation}
Define
\begin{align*}
A=\bigg\{\omega\in \Omega ~:~ &\left|\int_0^{t}\sqrt{\delta}\left(Y^{\eps,\delta}_{x, i}(s)-\xi^{\eps}_{x, i}(s)\right)'\sigma_h(Y^{\eps,\delta}_{x, i}(s),\alpha^\eps(s))dW(s)\right|\\
&\qquad-\dfrac1\delta\int_0^{t}\delta\left|Y^{\eps,\delta}_{x, i}(s)-\xi^{\eps}_{x, i}(s)\right|^2\left\|\sigma_h\sigma_h'(Y^{\eps,\delta}_{x, i}(s),\alpha^\eps(s))\right\|ds\leq k_1\, \text{ for all }\, t\in[0,T]\bigg\}.
\end{align*}
By the exponential martingale inequality, we get that for any $\delta<k_1$
\begin{align*}
\PP(A)\geq 1-2\exp\left(-\dfrac{2k_1}\delta\right)\geq1-\exp\left(-\frac{k_1}\delta\right).
\end{align*}
Since $f_h$ is Lipschitz and $\sigma_h$ is bounded, there is
an $M_1>0$ such that for all $\omega\in A$,
\begin{equation}
\begin{aligned}
|Y^{\eps,\delta}_{x, i}&(t)-\xi^{\eps}_{x, i}(t)|^2\\
\leq & 2 \int_0^{t}|Y^{\eps,\delta}_{x, i}(s)-\xi^{\eps}_{x, i}(s)||f_h(Y^{\eps,\delta}_{x, i}(s),\alpha^\eps(s))-f_h(\xi^{\eps}_{x, i}(s),\alpha^\eps(s))|ds\\
&\ +2\int_0^{t}\big|Y^{\eps,\delta}_{x, i}(s)-\xi^{\eps}_{x, i}(s)\big|^2\|\sigma_h\sigma_h'(Y^{\eps,\delta}_{x, i}(s),\alpha^\eps(s)\big)\|ds\\
&\ +\int_0^{t}\delta\trace\big((\sigma_h\sigma_h')(Y^{\eps,\delta}_{x, i}(s),\alpha^\eps(s))\big)ds+2\int_0^tk_1ds\\
\leq & M_1\int_0^t|Y^{\eps,\delta}_{x, i}(t)-\xi^{\eps}_{x, i}(t)|^2ds+(2k_1+M_1\delta) t
.\end{aligned}
\end{equation}
For each $t\in[0, T]$, an application of Gronwall's inequality implies that on the set $A$,
$$|Y^{\eps,\delta}_{x, i}(t)-\xi^{\eps}_{x, i}(t)|^2\leq (2k_1+M_1\delta)T\exp(M_1T)<\gamma^2$$ for $0<\delta<k_1$ sufficiently small.
It also follows from this inequality that for $\omega\in A$ and $0<\delta<k_1$ sufficient small, we have $\zeta>T$, which implies $$|X^{\eps,\delta}_{x, i}(t)-\xi^{\eps}_{x, i}(t)|^2=|Y^{\eps,\delta}_{x, i}(t)-\xi^{\eps}_{x, i}(t)|^2<\gamma^2,$$ for all $ t\in[0,T]$.
\end{proof}

\begin{lm}\label{lm2.1b}
For each $x$ and $\gamma$, we can find $k_{\gamma,x}=k_{\gamma,x}(T)>0$ such that
$$\PP\left\{\left|\xi^{\eps}_{x, i}(t)-\bar X_x(t)\right|\geq\gamma ~\text{for some}~t\in [0,T]\right\}\leq\exp\left(-\frac{k_{\gamma,x}}\eps\right),$$
where $\bar X_x(t)$ is the solution to equation \eqref{eq2.2} with the initial value $x$.
\end{lm}

\begin{proof}
This follows from the large deviation principle shown in \cite{HYZ}. We note that the existence and boundedness of a unique solution to equation \eqref{eq2.3} follows from parts (i) and (ii) of Assumption \ref{asp1}.
\end{proof}

By combining the results of Lemmas \ref{lm2.1} and \ref{lm2.1b} we can
%obtain
prove
Lemma \ref{lm2.2}.

\begin{proof}[Proof of Lemma \ref{lm2.2}]
By virtue of Lemma \ref{lm2.1b}, for each $x$ and $\gamma$, we have
$$\PP\left\{\left|\xi^{\eps}_{x, i}(t)-\bar X_x(t)\right|\geq\dfrac\gamma6  ~\text{for some}~t\in [0,T]\right\}\leq\exp\left(-\frac{k_{\gamma/6,x}}\eps\right).$$
By part $(ii)$ of Assumption \ref{asp1} together with the Lyapunov method for \eqref{eq2.3}, we can find $H_{R, T}>0$ such that $|\xi^{\eps}_{x, i}(t)|\leq H_{R, T}$ and $|\bar X_x(t)|\leq H_{R, T}$ for all $|x|\leq R$ and $0\leq t\leq T$. Since $f(\cdot, i)$ is locally Lipschitz for all $i\in\M$, there is a constant $M_2>0$ such that $|f(u, i)-f(v, i)|\leq M_2|u-v|$ for all $|u|\vee|v|\leq H_{R, T}$ and $i\in\M$. Using the Gronwall
inequality, we have for $|x|\vee|y|\leq R$, $i\in\M$ and any $t\in [0,T]$
\bea \ad |\xi^{\eps}_{x, i}(t)-\xi^{\eps}_{y,i}(t)|\leq|x-y|\exp(M_2T),\\
\ad
|\bar X_x(t)-\bar X_y(t)|\leq|x-y|\exp(M_2T).\eea
Let $\lambda=\dfrac{\gamma}6\exp(-M_2T)$. It is easy to see that for $|x-y|<\lambda$,
$$\PP\left\{\left|\xi^{\eps}_{y,i}(t)-\bar X_y(t)\right|\geq\dfrac\gamma2 ~\text{for some}~t\in [0,T] \right\}\leq\exp\Big(-\frac{k_{\gamma/6,x}}\eps\Big).$$
By the compactness of $B_R$, for $\gamma>0$, there is $k_2=k_2(R, T,\gamma)>0$ such that for all $x\in B_R$,
$$\PP\left\{\left|\xi^{\eps}_{x, i}(t)-\bar X_x(t)\right|\geq\dfrac\gamma2 ~\text{for some}~t\in [0,T]\right\}\leq \exp\left(-\frac{k_2}\eps\right).$$
Combining this with Lemma \ref{lm2.1}, we have
\bea \disp \PP\left\{\left|X^{\eps,\delta}_{x,i}(t)-\bar X_x(t)\right|\geq\gamma ~\text{for some}~t\in [0,T] \right\}\ad <\exp\left(-\frac{k_1(R, T,\gamma/2)}\delta\right)+\exp\left(-\frac{k_2}\eps\right)\\
\ad <\exp\left(-\frac{\kappa}{\eps+\delta}\right)\eea
for a suitable $\kappa=\kappa(R,T,\gamma)$ and for all sufficiently small $\eps$ and $\delta$.
\end{proof}

\begin{proof}[Proof of Lemma \ref{lm3.1}]
Let $n^{\eps,\delta}\in\N$ such that $n^{\eps,\delta}-1<\dfrac{1}{a^{\eps,\delta}}\leq n^{\eps,\delta}.$
We consider events $A_k=\{X_{x,i}^{\eps, \delta}(t)\in N, \ \forall (k-1)\ell <t\leq k\ell \}$.
We have $\PP(A_1)\leq1-a^{\eps,\delta}.$
By the Markov property,
\bea \disp\PP(A_k|A_1,...,A_{k-1})\ad =\int_{N}\PP\left\{\check\tau_y^{\eps, \delta}\leq \ell \right\}\PP\Big\{X_{x,i}^{\eps, \delta}((k-1)\ell )\in dy\Big|A_1,...,A_{k-1}\Big\}\\
\ad\leq1-a^{\eps,\delta}.\eea
As a result,
$$\PP(A_1A_2\cdots A_n)\leq (1-a^{\eps,\delta})^{n^{\eps,\delta}}$$
Since $\lim\limits_{\eps\to0}a^{\eps,\delta}=0$, we deduce that $\lim\limits_{\eps\to0}(1-a^{\eps,\delta})^{n^{\eps,\delta}}=e^{-1}$, which means that $(1-a^{\eps,\delta})^{n^{\eps,\delta}}<1/2$ for sufficiently small $\eps$.
\end{proof}
\begin{proof}[Proof of Lemma \ref{lm2.5}]
The proof is omitted because
it states some standard properties of dynamical systems.
Interested readers can refer to \cite{perko13}.
\end{proof}
\end{document}